\documentclass[12pt]{amsart}

\usepackage{amsmath}
\usepackage{caption}
\usepackage[curve]{xypic}
\usepackage{cite}
\usepackage{enumitem}
\usepackage{color}
\usepackage{url}
\usepackage[usenames,dvipsnames]{xcolor}
\usepackage{graphicx}
\usepackage{bm}
\usepackage[figure,norelsize]{algorithm2e}

\makeatletter 
\def\@cite#1#2{{\m@th\upshape\bfseries%
[{#1\if@tempswa{\m@th\upshape\mdseries, #2}\fi}]}}
\makeatother 

\theoremstyle{plain}
\newtheorem{thm}{Theorem}[section]
\newtheorem{cor}[thm]{Corollary}
\newtheorem{prop}[thm]{Proposition}
\newtheorem{lem}[thm]{Lemma}

\newtheorem{ass}[thm]{Assumption}
\theoremstyle{definition}

\newtheorem{quest}[thm]{Question}
\newtheorem{prob}[thm]{Problem}

\theoremstyle{remark}
\newtheorem{rem}[thm]{Remark}

\numberwithin{equation}{subsection}
\renewcommand{\bold}[1]{\medskip \noindent {\bf #1 }\nopagebreak}

\newcommand{\nc}{\newcommand}
\newcommand{\rnc}{\renewcommand}
\newcommand{\e}{\varepsilon}

\newcommand{\ann}[1]{\marginpar{\tiny{}}}

\nc\bA{\mathbb{A}}
\nc\bB{\mathbb{B}}
\nc\bC{\mathbb{C}}
\nc\bD{\mathbb{D}}
\nc\bE{\mathbb{E}}
\nc\bF{\mathbb{F}}
\nc\bG{\mathbb{G}}
\nc\bH{\mathbb{H}}
\nc\bI{\mathbb{I}}
\nc{\bJ}{\mathbb{J}} 
\nc\bK{\mathbb{K}}
\nc\bL{\mathbb{L}}
\nc\bM{\mathbb{M}}
\nc\bN{\mathbb{N}}
\nc\bO{\mathbb{O}}
\nc\bP{\mathbb{P}}
\nc\bQ{\mathbb{Q}}
\nc\bR{\mathbb{R}}
\nc\bS{\mathbb{S}}
\nc\bT{\mathbb{T}}
\nc\bU{\mathbb{U}}
\nc\bV{\mathbb{V}}
\nc\bW{\mathbb{W}}
\nc\bY{\mathbb{Y}}
\nc\bX{\mathbb{X}}
\nc\bZ{\mathbb{Z}}
\nc\cA{\mathcal{A}}
\nc\cB{\mathcal{B}}
\nc\cC{\mathcal{C}}
\rnc\cD{\mathcal{D}}
\nc\cE{\mathcal{E}}
\nc\cF{\mathcal{F}}
\nc\cG{\mathcal{G}}
\rnc\cH{\mathcal{H}}
\nc\cI{\mathcal{I}}
\nc{\cJ}{\mathcal{J}} 
\nc\cK{\mathcal{K}}
\rnc\cL{\mathcal{L}}
\nc\cM{\mathcal{M}}
\nc\cN{\mathcal{N}}
\nc\cO{\mathcal{O}}
\nc\cP{\mathcal{P}}
\nc\cQ{\mathcal{Q}}
\rnc\cR{\mathcal{R}}
\nc\cS{\mathcal{S}}
\nc\cT{\mathcal{T}}
\nc\cU{\mathcal{U}}
\nc\cV{\mathcal{V}}
\nc\cW{\mathcal{W}}
\nc\cY{\mathcal{Y}}
\nc\cX{\mathcal{X}}
\nc\cZ{\mathcal{Z}}

\newcommand{\bk}{{\mathbf{k}}}

\nc{\dmo}{\DeclareMathOperator}
\rnc{\Re}{\operatorname{Re}}
\rnc{\Im}{\operatorname{Im}}
\dmo{\rank}{rank}
\dmo{\End}{End}
\dmo{\Hom}{Hom}
\dmo{\Jac}{Jac}
\dmo{\Id}{Id}
\dmo{\Ann}{Ann}
\dmo{\Area}{Area}
\dmo{\CP}{\bP^1}

\title{Full rank affine invariant submanifolds}
\author[Mirzakhani]{Maryam~Mirzakhani}
\author[Wright]{Alex~Wright}
\begin{document}
\maketitle
\thispagestyle{empty}


\vspace{1cm}



\section{Introduction}\label{S:intro}

\ann{Code: ``R1-R4" denotes referee reports 1-4. Report 1 was a preliminary report by the author of report 3. Reports 2, 3, 4 have 6, 1, and 4 pages respectively. All comments other than typos have been added as annotations. ``A" denotes the response of the authors and marks all changes the authors made. All comments by all referees have been addressed.} 
\bold{General context.} A translation surface is a pair $(X, \omega)$, where $X$ is a Riemann surface and $\omega$ is a nonzero Abelian differential on $X$. The moduli space of translation surfaces of fixed genus is the complement of the zero section of the Hodge bundle over the  moduli space of Riemann surfaces, and admits a $GL(2, \bR)$ action. The study of this action  grew out of  problems in Teichm\"uller theory, and dynamical systems arising in physics such as billiards in a polygon, and has been enriched in the past decades by connections to ergodic theory on homogeneous spaces and  algebraic geometry.

The relation to billiards is given via a  procedure called unfolding, which associates a translation surface 
to each polygon whose angles are rational multiples of $\pi$. However, many concrete applications to billiards have been elusive, because the set of unfoldings of polygons is measure zero. Thus, despite the fact that almost every translation surface has dense $GL(2, \bR)$ orbit, the $GL(2, \bR)$ orbit closures of unfoldings have remained mysterious. 

One application of our main result will be to show that infinitely many triangles unfold to translation surfaces with dense $GL(2,\bR)$ orbit. Besides  polygons all of whose angles are multiples of $\frac\pi2$ and polygons that unfold to translation surfaces of low genus or with closed orbit, these are the first examples  where the orbit closure of the unfolding of\ann{A: Typo found by R4 fixed.} a polygon has been computed. Our proof uses new results explained below, and builds upon a wide array of technology that has been developed over the past decades by a large and diverse group of mathematicians.

Recently discovered examples of\ann{A: Typo found by R2 fixed.} $GL(2, \bR)$ orbit closures have lead to the absence even of conjectures on what types of $GL(2, \bR)$ orbit closures exist. Our intuition is that very large non-trivial orbit closures should not exist, and largeness should be measured by rank, an integer between 1 and the genus. Our main result, which classifies maximal rank orbit closures, is a first step in this direction. The proof  proceeds by inductive arguments, and marked points produced using degenerations are remembered and made use of.  We hope these techniques will be applicable in future work.

\bold{Statement of main results.} The Hodge bundle is stratified according to the number and order of the zeros, and this stratification is preserved by the $GL(2, \bR)$ action. 
Following \cite{EM, EMM}, orbit closures for the $GL(2, \bR)$ action are also known as affine invariant submanifolds. We say that an affine invariant submanifold has full rank if its rank has maximal value, that is, the rank is equal to the genus. Rank is defined as one half of the dimension of the  image of the tangent space in absolute\ann{A: Typo found by R2 fixed.} cohomology (see Section \ref{S:background}).  
So $\cM$ has full rank if and only if the absolute periods of Abelian differentials are locally unconstrained in $\cM$. 

We prove that \ann{A: Typo found by R2 fixed.} non-trivial full rank affine invariant submanifolds do not exist.  

\begin{thm}\label{T:main}
Let $\cM$ be a full rank affine invariant submanifold. Then $\cM$ is either a connected component of a stratum, or the locus of hyperelliptic translation surfaces therein. 
\end{thm}

By work of Filip \cite[Theorem 1.6]{Fi2}, this is equivalent to the following. 



\begin{thm}\label{T:end}
If $\cM$ is any affine invariant submanifold other than a connected component of a stratum or a hyperelliptic locus, then the Jacobian of every translation surface in $\cM$ has non-trivial endomorphisms. 
\end{thm}

 Theorem \ref{T:main} has applications to unfoldings of polygons. 

\begin{thm}\label{T:billiards}
Given any polygon all of whose angles are multiples of $\frac\pi3$, almost every polygon with the same angles unfolds to a translation surface whose orbit is dense in a connected component of a stratum. 
\end{thm}

\begin{thm}\label{T:triangles}
There are infinitely many rational triangles whose unfoldings have dense $GL(2, \bR)$ orbit in a connected component of a stratum. 
\end{thm}

Theorem \ref{T:triangles} is in contrast to the fact that unfoldings of triangles have many properties in common with translation surfaces with \emph{closed} $GL(2,\bR)$ orbits. For example, their Jacobians have a factor with real (and even complex) multiplication, and the difference of any two zeros of the Abelian differential is torsion in the Jacobian, compare to \cite{M, M2}.  

The greatest common denominator $k$ of the angles divided by $\pi$ plays a prominent role in the study of the unfolding of a rational polygon. The methods used to prove Theorem \ref{T:triangles} show that at least 74 percent of the 1436 non-isosceles triangles with $k$ odd and less than  $50$  unfold to translation surfaces with dense orbit. See Figure \ref{L:list} for a list when $k$ is odd and at most 25. When $k$ is even the unfolding has an involution, and in some cases when  this involution is hyperelliptic  the methods can show the orbit is dense in the hyperelliptic locus. 

The proof of Theorem \ref{T:triangles} also uses \ann{A: Typo found by R4 fixed.} work of Filip on variation of Hodge structures, and often gives non-trivial lower bounds for the rank of the orbit closure of the unfolding even in instances where the methods do not show the orbit closure has full rank. (See Theorem \ref{T:unstab}.) 
Since closed orbits are rank 1 orbit closures, our methods are  applicable to the open problem of  classifying obtuse triangles that unfold to translation surfaces with closed orbits \cite{H2, KS, Pu}.

\bold{A new result on cylinder deformations.} Every translation surface has infinitely many cylinders \cite{MasurClosed}, and the cylinder deformations introduced in \cite{Wcyl} can be used to study  affine invariant submanifolds. A cylinder on a translation surface in an affine invariant submanifold $\cM$ is called $\cM$-free if the result of stretching or shearing this cylinder remains in $\cM$. The following is a key ingredient in the proof of Theorem \ref{T:main}. It says that the only affine invariant submanifolds that are unconstrained in terms of cylinder deformations are connected components of strata.   

\begin{thm}\label{T:free}
Suppose that $\cM$ is an affine invariant submanifold for which every cylinder on every surface in $\cM$ is $\cM$-free. Then $\cM$ is a connected component of a stratum. 
\end{thm}

This implies in particular that if a surface has no pairs of parallel cylinders then its orbit is dense. 

%
%


\bold{Theorems \ref{T:main} and \ref{T:free} are false for quadratic differentials.} For any connected component $\cQ$ of a stratum of quadratic differentials, consider the locus $\tilde{\cQ}$ of  Abelian differentials obtained via double covers of all quadratic differentials in $\cQ$. In analogy with  Theorems \ref{T:main} and \ref{T:free}, one can ask the following questions. 

\begin{enumerate}[label=(\Alph*)]
\item Is every affine invariant submanifold $\cM\subset \tilde{\cQ}$ with $\rank\cM= \rank\tilde{\cQ}$ equal to $\tilde{\cQ}$?
\item Suppose  $\cM\subset \tilde{\cQ}$ is an affine invariant submanifold, and every cylinder deformation of every surface in $\cM$ that remains in $\tilde{\cQ}$ in fact remains in $\cM$.  Must $\cM= \tilde{\cQ}$?
\end{enumerate}

%
%

Both questions have negative answers.  The first  will be addressed in forthcoming work of Eskin-McMullen-Mukamel-Wright, which gives examples of unfoldings of polygons with unexpected orbit closures. The existence of  hyperelliptic square-tiled surfaces with exactly two cylinders (which are homologous) in any cylinder direction implies a negative answer to the second question \cite[Remark 3.1]{MY}.


\bold{Relation to previous results.} McMullen classified orbit closures of translation surfaces in genus 2 \cite{Mc5}. (One open problem remains concerning square-tiled surfaces.) Apisa classified higher rank (rank larger than 1) orbit closures in hyperelliptic connected components of strata \cite{Apisa}. In  hyperelliptic connected components of strata Theorem \ref{T:main} follows from\ann{A: Typo found by R2 and R4 fixed.} their work. In  strata of Abelian differentials with only one zero, Theorem \ref{T:main} is immediate from the definition of rank. 

Rank 2 affine invariant submanifolds in strata of genus 3 Abelian differentials with 2 zeros were classified in \cite{AN}, and so Theorem \ref{T:main} completes the classification of higher rank  affine invariant submanifolds in these strata. Higher rank affine invariant submanifolds of genus 3 Abelian differentials with only one zero were classified in \cite{NW, ANW}. We hope Theorem \ref{T:main} can play a  role in ongoing efforts to classify higher rank affine invariant submanifolds in the remaining strata in genus 3. 

Theorems \ref{T:main} and \ref{T:end} were claimed in the special case of strata of Abelian differentials with $2g-2$ simple zeros  in \cite{M6}, however there was a gap in the proof.\footnote{According to M\"oller, the surjectivity claim in the second paragraph of the proof of \cite[Theorem 4.1]{M6} is not justified.}

Athreya-Eskin-Zorich studied unfoldings of right angled polygons, and in particular showed that the generic such unfolding has dense orbit in the hyperelliptic locus \cite{AEZ:right}. Their methods do not apply to polygons whose angles are not all multiples of $\frac\pi2$, but also give stronger conclusions. 

In general, unfoldings of triangles may have rank 1 orbit closure \cite{V}. Moreover, at first glance it may seem plausible that all unfoldings of triangles have rank 1 orbit closure. To apply Theorem \ref{T:main} to unfoldings of triangles, we rely on the connection between $GL(2,\bR)$ orbit closures and Jacobians with extra endomorphisms, due to McMullen, M\"oller, and Filip \cite{Mc, M, Fi2}. We were inspired by McMullen's proof that the locus of eigenforms for real multiplication is not $GL(2, \bR)$ invariant in genus 3, and we  use  Filip's recent results on the variation of Hodge structure over an affine invariant submanifold. These results of Filip were a key component of his proof that affine invariant submanifolds are algebraic varieties \cite{Fi1}.  

New orbit closures related to unfoldings of quadrilaterals are given in \cite{MMW} and in forthcoming work of Eskin-McMullen-Mukamel-Wright. 


\bold{Outline of the paper.}  Section \ref{S:background} recalls necessary background on cylinder deformations and affine invariant submanifolds.

Theorem \ref{T:main} asserts that  an affine invariant submanifold $\cM$ of full rank that is not equal to a connected component of a stratum must be the hyperelliptic locus. We assume by induction that this result has been proven for all strata of surfaces with smaller genus or same genus but with fewer zeros. The base case is genus 2. 

Section \ref{S:prep} prepares for the proof of Theorem \ref{T:main} by using Theorem \ref{T:free} to show that there is a collection of homologous cylinders on a surface $(X, \omega)$ in $\cM$ such that not all these cylinders are free.  This collection is later shown to consist of exactly two homologous cylinders, so for simplicity we will assume this now. We think of this pair of cylinders as a certificate that $\cM$ is not a connected component of a stratum.

Section \ref{S:main} completes the proof of Theorem \ref{T:main}. The two homologous cylinders divide the surface into two halves, and  a degeneration supported in each  gives rise to two affine invariant submanifolds $\cM_1, \cM_2$ that lie in the boundary of $\cM$ described in  \cite{MirWri}.  It is not obvious that these $\cM_i$ are full rank (see Figure \ref{F:RelToAbs}), but we are able to verify this and moreover show that neither is a non-hyperelliptic connected component of  a stratum. The inductive hypothesis then gives that the $\cM_i$ are hyperelliptic,  and we conclude by combining the hyperelliptic involutions on the two degenerations to create a hyperelliptic involution on $(X, \omega)\in \cM$. 

Both the proof that the $\cM_i$ are not non-hyperelliptic connected component of strata as well as the step which combines the hyperelliptic involutions make use of results of Apisa and Apisa-Wright on marked points. Some technical difficulties complicate the proof of Theorem \ref{T:main}, so that the above  rough outline of the idea of the proof is not completely accurate. 

Section \ref{S:free} proves  Theorem \ref{T:free} by making use of deformations arising from\ann{A: Typo found by R2 fixed.} cylinders on different surfaces to obtain horizontally periodic  surfaces with maximally many horizontal cylinders.  

Section \ref{S:billiards} proves that for any tuple of rational angles there is an affine invariant submanifold which is equal to the orbit closure of the unfolding of a generic polygon with those angles, and establishes Theorem \ref{T:billiards}. Section \ref{S:deriv} makes use of variation of Hodge structures to prove Theorem \ref{T:triangles}. 


\bold{Open problems.} In  Section \ref{S:billiards}, we associate an affine invariant submanifold\ann{A: Typo found by R4 fixed.} $\cM(\bm\theta)$ to any tuple $\bm\theta=(\theta_1, \ldots, \theta_n)$ of angles of a rational polygon; this can be defined as the orbit closure of the unfolding of the generic polygon with these angles. One long term goal is to compute these $\cM(\bm\theta)$ for all $\bm\theta$. 

\begin{quest}
If $n=4$, are there infinitely many $\cM(\bm\theta)$ such that $\bm\theta$ has more than two distinct angles and $\cM(\bm\theta)$ is neither a connected component of a stratum nor a locus $\tilde{\cQ}$ of double covers of all quadratic differentials in a given stratum? If $n>4$, are there any at all?
\end{quest}\ann{R4: Question 1.6. ... is neither a connected component of a stratum nor one of the loci $\tilde{\cQ}$ described above.
It would be better to repeat explicitly, which loci do you mean. You described above various different loci.\\A: Included definition of $\tilde{\cQ}$.}

%
%
%

In relation to Theorem \ref{T:free} and Question B (the analogue of Theorem \ref{T:free} for quadratic differentials) we have the following. 


\begin{quest} 
Suppose that $\cM$ is a higher rank affine invariant submanifold, and any pair of $\cM$-parallel cylinders on any $(X,\omega)\in \cM$ are homologous.  Must $\cM$ be a connected component of a stratum or a hyperelliptic locus? 
\end{quest}\ann{R4: Question 1.7. Suppose that $\cM$ is a higher rank affine invariant sub- manifold, and any pair...
I suggest to be formal and to insert ``in $\tilde{\cQ}$".\\A: We do not wish to assume $\cM\subset \tilde{\cQ}$.}


\begin{quest}
Suppose  $\cM\subset \tilde{\cQ}$ is a higher rank affine invariant submanifold, and every cylinder deformation of every surface in $\cM$ that remains in $\tilde{\cQ}$ in fact remains in $\cM$. Must $\cM= \tilde{\cQ}$?
\end{quest}

\bold{Acknowledgements.} We thank Paul Apisa, Alex Eskin, Jeremy Kahn, Carlos Matheus, Curtis McMullen, Ronen Mukamel, Duc-Manh Nguyen, Barak Weiss, and Anton Zorich  for helpful conversations. We thank the referees for their helpful  comments.\ann{A: Added thanks to referees.} %
The work of MM was partially supported by NSF and Simons grants.
The work of AW was partially supported by a Clay Research Fellowship.

\section{Background and notation}\label{S:background}

\bold{Definitions.} The stratum $\cH(k_1, \ldots, k_s)$ is defined to be the set of translation surfaces $(X, \omega)$, where $\omega$ has $s$ zeros, of orders $k_1, \ldots, k_s$. Connected components of strata are classified in \cite{KZ:comps}. We allow marked points, which are considered to be zeros of order zero. However, unless otherwise indicated, surfaces in this paper will not have marked points.

An affine invariant submanifold is a properly immersed submanifold of a stratum whose image is given in local period coordinates near any point by a finite union of real linear subspaces.  
 \ann{R1: Define affine invariant\\A: We added the definition.}

The rank of an affine invariant submanifold $\cM$ is defined as 
$$\rank(\cM)=\frac12 \dim_\bC p(T_{(X, \omega)}\cM).$$ Here $(X, \omega)$ is any surface in $\cM$, and $T_{(X, \omega)}\cM\subset H^1(X, \Sigma)$ is the tangent space to $\cM$ at $(X, \omega)$. As usual, $\Sigma$ denotes the collection of  marked points and  zeros of $\omega$, and $p:H^1(X, \Sigma)\to H^1(X)$ is the usual map from relative to absolute cohomology with complex coefficients. Rank is an integer because a result of Avila-Eskin-M\"oller gives that $p(T_{(X, \omega)}\cM)$ is symplectic \cite{AEM}. Note that $\cM$ has full rank if and only if $p(T_{(X, \omega)}\cM)=H^1(X)$. 

\bold{Cylinder deformations.} A cylinder on a translation surface is the image of an isometric embedding of $\bR/(c\bZ) \times (0,h)$ that does not have marked points in its interior and is not a proper subset of any other cylinder. We call $c$ the circumference, $h$  the height, and $h/c$  the modulus. 

Let $\cM$ be an affine invariant submanifold, and consider a surface $(X,\omega)\in \cM$. Two cylinders on $(X,\omega)$ are said to be $\cM$-parallel if they are parallel and remain so on any sufficiently nearby surface in $\cM$. The ratio of circumferences of two $\cM$-parallel cylinders remains constant on all nearby surfaces in $\cM$. (See \cite{LNW} for details on the case when $\cM$ has self-crossings.) Indeed, if $\alpha_1, \alpha_2$ are the core curves and $c$ is the ratio of circumferences, then $\int_{\alpha_1} \omega = c\int_{\alpha_2} \omega$ is one of the linear equations that locally define $\cM$. 

Define $$u_t=\left(\begin{array}{cc} 1&t\\0&1\end{array}\right), \quad \quad a_t=\left(\begin{array}{cc}1&0\\0& e^t\end{array}\right).$$ 
Given a cylinder $C$ on a translation surface $(X, \omega)$ with a choice of orientation for the core curve,  the cylinder shear (or twist) $u_t^C(X, \omega)$ is defined to be the result of rotating $(X, \omega)$ so that $C$ becomes horizontal with core curve oriented in the positive real direction, applying the matrix $u_t$ only to $C$, and then applying the inverse rotation. The cylinder stretch $a_t^C(X, \omega)$ is  defined similarly. Using the Poincare duality isomorphism between $H^1(X, \Sigma)$ and $H_1(X\setminus \Sigma)$, these deformations are dual to scalar multiples of the core curve of $C$ \cite[Section 4]{MirWri}.

A collection of parallel cylinders is called consistently oriented if an orientation for each of their core curves is chosen and the integral of $\omega$ over any two of these curves are positive multiples of each other. If $\cC=\{C_1, \ldots, C_k\}$ is a collection of parallel consistently oriented cylinders, we define $u_t^\cC(X,\omega)= u_t^{C_1}  \cdots u_t^{C_k}(X,\omega)$ and $a_t^\cC(X,\omega)= a_t^{C_1}  \cdots a_t^{C_k}(X,\omega)$. The main result of \cite{Wcyl} is the following. 

\begin{thm}[Cylinder Deformation Theorem]\label{T:CDT}
If $\cC$ is an equivalence class of $\cM$-parallel cylinders on $(X,\omega)$, then $u_t^\cC  a_s^\cC(X,\omega) \in\cM$ for all $s,t\in \bR$. 
\end{thm}

In particular, the derivative $u^\cC$ of $u^\cC_t(X,\omega)$ lies in $T_{(X,\omega)}\cM$. 
A key ingredient  in\ann{A: Typo found by R2 and R4 fixed.} the proof is the following result of \cite{SW2}, which builds upon the work of Minsky-Weiss\ann{A: Typo found by R2 fixed.} \cite{MinW}. 

\begin{thm}[Smillie-Weiss]\label{T:SW}
Every closed, horocycle flow invariant subset of a stratum contains a horizontally periodic surface. 
\end{thm}

%
%

\bold{Marked points.}  Degenerating a surface without marked points may produce a surface with marked points. Because of this, the study of orbit closures of surfaces with marked points cannot be separated from the case with no marked points. 

Let $\cM$ be an affine invariant submanifold of a stratum $\cH$. Let $\cH^{n*}$ be the stratum of surfaces in $\cM$ together with a collection of $n$ distinct marked points, none of them equal to a zero of positive order\ann{A: Typo found by R2 fixed.}, and let $\pi:\cH^{n*}\to \cH$ be the natural forgetful map.  

We define an $n$-point marking over $\cM$ to be an affine invariant submanifold $\cN$ of the stratum $\cH^{*n}$ such that $\pi(\cN)$ is equal to a dense subset of $\cM$, equivalently, $\pi(\cN)$ is equal to $\cM$ minus a finite, possibly empty, union of smaller dimensional affine invariant submanifolds. 

The proof of Theorem \ref{T:main} uses the following result of Apisa  \cite{Apisa}.\ann{A: Typo found by R2 fixed.}

\begin{thm}[Apisa]\label{T:apisa}
Suppose that $\cH$ is a non-hyperelliptic connected component of a stratum. Any $n$-point marking $\cN$ over $\cH$ is given by $n$ unconstrained points, i.e., $\dim \cN = \dim \cH+n$. 
\end{thm}

The following extension of this result, established by Apisa \cite{Apisa2} for hyperelliptic connected components of strata and in general in forthcoming work, states that there are no non-trivial point markings over the hyperelliptic locus.   

\begin{thm}[Apisa, Apisa-Wright]\label{T:marked}
Suppose that $\cN$ is an $n$-point marking over the hyperelliptic locus $\cM$ in some stratum. Then there are non-negative integers $f,w,$ and $v$ such that $n=f+w+2v$  and 
$\cN$ is the set of all $(X,\omega, S)$ where $|S|=n$, at least $w$ points of $S$ are fixed by the hyperelliptic involution, and at least $v$ pairs of points of $S$ are exchanged by the hyperelliptic involution. 
\end{thm} 

Note $\dim \cN = \dim \cH+f+v$, and $f$ of the points of $S$ are unconstrained.

\section{Preparation for the proof of Theorem \ref{T:main}}\label{S:prep}

The results of this section and the next  assume Theorem \ref{T:free}, which is proven in Section \ref{S:free}.

\ann{R1: Make clear that proofs depend on Theorem 1.5 which will be proven at end\\A: We added a sentence saying we are using Theorem 1.5, and saying where to find the proof of Theorem 1.5.}
\begin{lem}\label{L:hom}
If $\cM$ has full rank and $(X,\omega)\in \cM$, then two cylinders on $(X, \omega)$ are $\cM$-parallel if and only if they are homologous. 
\end{lem}

\ann{R2: Proof of Lemma 3.1. Consider mentioning that the other direction is trivial.\\A: We added a remark that the other direction is trivial.}
\begin{proof}
It is trivial that homologous cylinders are $\cM$-parallel (for any affine invariant submanifold $\cM$). 

If $\alpha_1$ and $\alpha_2$ are core curves of $\cM$-parallel cylinders, then the equation $\int_{\alpha_1} \omega = c\int_{\alpha_2} \omega$ must hold locally on $\cM$, where $c$ is the ratio of circumferences \cite[
Lemma 4.7]{Wcyl}. However, no non-trivial linear equations on absolute periods hold for a full rank affine invariant submanifold, so $\alpha_1$ and $\alpha_2$ must be homologous. 
\end{proof}

\begin{lem}\label{L:genus2}
Theorem \ref{T:main} holds in genus 2.
\end{lem}

This is a consequence of \cite{Mc5}, and also follows from Theorem \ref{T:free} as follows.
 
\begin{proof}
A genus 2 translation surface never has a pair of  homologous cylinders, so the Cylinder Deformation Theorem and Lemma \ref{L:hom} imply that every cylinder on every surface  is free. 
\end{proof}

\begin{rem}
Similarly, Theorem \ref{T:main} follows from Theorem \ref{T:free} in strata where the Abelian differentials have two zeros, both of odd order, because such Abelian differentials may not have homologous cylinders.
\end{rem}\ann{R1: Can you explain Remark 3.3\\A: We added a phrase of extra explanation, however we prefer not to get into details since this is just an offhand remark that is not used in the paper. (However we did provide a detailed explanation to you in our response to your initial comments.)}

\ann{R4: Proposition 3.4.
You prove this Proposition conditionally to Theorem 1.5. I suggest to state this explicitly (and other assumptions, if there are further conditional assumptions.) I realize, that it is mentioned in the ``outline of the paper", but still, it would make easier to follow the logic of the proof step-by-step.\\A: We clarified this with a sentence at the beginning of this section.}
\begin{prop}\label{P:setup}
Assume $\cM$ is full rank but not a connected component of a stratum. 
Then there exists $(X,\omega)\in \cM$ such that: 
\begin{enumerate}
\item $(X,\omega)$ has a collection $\cC=\{C_1, \ldots, C_k\}$ of horizontal homologous cylinders, such that the heights of the cylinders in $\cC$ are locally constrained in $\cM$ to satisfy a non-trivial linear equation, and such that $\cC$ includes any cylinder homologous to  $C_1$; 
\item $(X,\omega)$ is horizontally periodic, and the core curves of the horizontal cylinders span a Lagrangian subspace of $H_1(X)$;
\item Every connected component of the complement of $\cup_{i=1}^k C_i$ contains a horizontal cylinder that is not homologous to any other  cylinder. 
\end{enumerate}
\end{prop} 

 In (1), non-trivial means that not all coefficients are zero. 

\begin{rem}\label{R1:notation}
The following notation will be used throughout this section and the next. Recall that a cylinder is an isometric embedding of $\bR/(c\bZ) \times (0,h)$. This always extends to a map from $\bR/(c\bZ) \times [0,h]$ to the surface. If the cylinder is horizontal, the map can be chosen to be locally of the form $(x+c\bZ,y) \mapsto x+iy$, where $\omega=d(x+iy)$. The top of a horizontal cylinder is defined to be the image of $\bR/(c\bZ) \times \{h\}$ and the bottom is defined to be the image of $\bR/(c\bZ) \times \{0\}$.

Given any collection of $k$ pairwise homologous disjoint simple closed curves on a surface, none of which is null-homologous, removing these curves gives $k$ connected components, each with exactly two boundary circles.

We will index the $C_i$  so that the connected component of the complement of $\cup_{i=1}^k C_i$ that shares a\ann{A: Typo found by R2 fixed.} boundary saddle connection with the top of $C_i$ also shares a boundary saddle connection with the bottom of $C_{i+1}$. \ann{R1: can't say the top of a cylinder $C_i$ share boundary with boundary of more than one other $C_j$? Is this ruled out by being homologous? Please explain\\ A: We added a sentence of explanation (in the previous paragraph).}
We will adopt the convention that $C_0=C_k$ and $C_{k+1}=C_1$. For each $i=1,\ldots, k$, we will pick once and for all one horizontal cylinder $D_i$  in the connected component of the complement of $\cup_{i=1}^k C_i$ that shares boundary with the top of $C_i$, such that $D_i$ is not homologous to any other cylinder. 
\end{rem}


\begin{lem}\label{L:section8}
Let $(X',\omega')$ be a surface in an affine invariant submanifold $\cM$ and let $\cD$ be the set of horizontal cylinders on $(X',\omega')$. Then there is a path in $\cM$ from $(X', \omega')$ to a surface $(X, \omega)$ such that all the cylinders in $\cD$ persist and remain horizontal along this path, and such that $(X,\omega)$ is horizontally periodic and the core curves of the horizontal cylinders span a subspace of $H_1(X)$ of dimension at least $\rank(\cM)$. 
\end{lem}

\begin{proof}
See \cite[Section 8]{Wcyl}. 
\end{proof}

\begin{proof}[Proof of Proposition \ref{P:setup}.]
By Theorem \ref{T:free}, there is a surface $(X',\omega')\in \cM$ with a cylinder $C_1$ that is not free. Let $\cC'$ be the set of cylinders  on $(X', \omega')$ homologous to $C_1$. 

Rotating the surface, we can assume the cylinders of $\cC'$ are horizontal. Let $(X, \omega)$ denote a surface as in Lemma  \ref{L:section8}, and let $\cC$ be the set of all\ann{A: Typo found by R2 fixed.} cylinders on $(X,\omega)$ homologous to cylinders that persist from cylinders in $\cC'$ on $(X', \omega')$. 

We first claim that not every cylinder in $\cC$ is free. Otherwise, for each cylinder $C'$ in $\cC'$, the corresponding cylinder $C$ in $\cC$ is free. Hence the Poincare dual of the core curve of $C$ is in $T_{(X,\omega)}\cM$. Parallel transporting this along the path to $(X', \omega')$ gives that the Poincare dual of the core curve of $C'$ is in $T_{(X',\omega')}\cM$, and hence that $C'$ is free.  Thus, because not every cylinder in $\cC'$ is free, not every cylinder in $\cC$ is free. 

Because the cylinders of $\cC$ are not all free, \cite[Corollary 3.4]{Wcyl} gives that some linear equation on the moduli of the cylinders of $\cC$ holds on all surfaces sufficiently close to $(X,\omega)$. See also \cite[Lemma 4.6]{MirWri}. Since the cylinders of $\cC$ all have equal circumference, this gives a linear equation on the heights. 

Since $\cM$ has full rank, this establishes the first two statements of the proposition. We will now show that there is a cylinder in every connected component of the complement of $\cup_{i=1}^k C_i$. Otherwise, because $(X,\omega)$ is horizontally periodic, the top of a cylinder of $\cC$, say $C_1$, must be glued to the bottom of another cylinder of $\cC$, say $C_2$, as in Figure \ref{F:C1C2}. \ann{R1: why should this imply all of top of $C_1$ glued to all of bottom $C_2$?\\ A: Added a phrase of explanation.}
\ann{R2: Proof of Proposition 3.4, last sentence: This is not obvious, so consider actually spelling this out. For example, you can use the following argument ...\\ A: We added the argument.}

\begin{figure}[h!]
\includegraphics[width=0.4\linewidth]{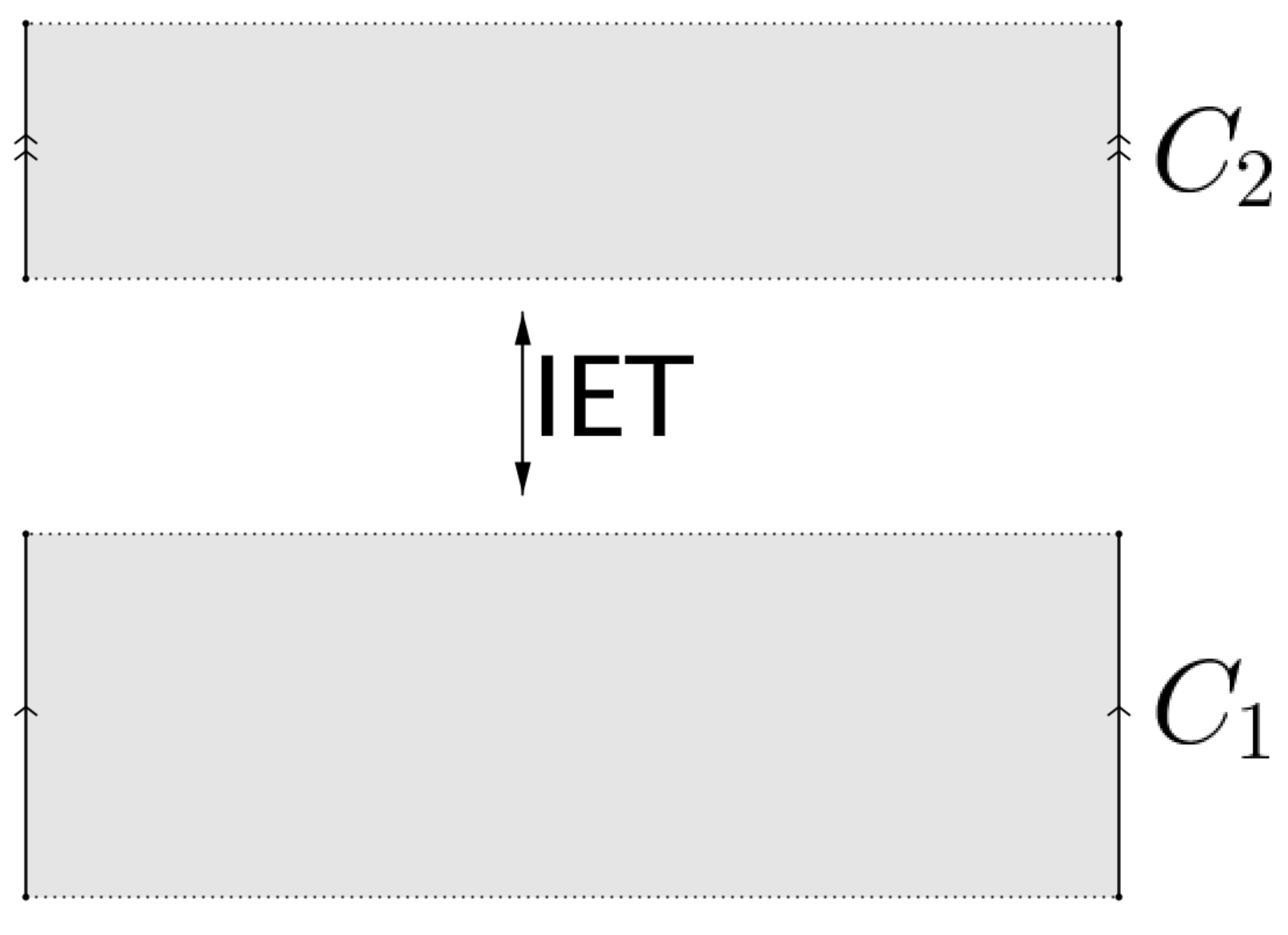}
\caption{Proof of Proposition \ref{P:setup}.}
\label{F:C1C2}
\end{figure}

The region above a core curve of $C_1$ and below a core curve of $C_2$ is a subsurface of positive genus, which contradicts the fact that the core curves of the horizontal cylinders span a Lagrangian subspace.\ann{A: Typo found by R4 fixed.}

Thus there is a horizontal cylinder $D_i$ in each connected component. 

 We must now show that we can pick $D_i$ to be not homologous to any other horizontal cylinder.
If the given connected component of the complement of $\cup_{i=1}^k C_i$ does not contain homologous cylinders, we are done. Otherwise, pick $E_1, E_2$ homologous cylinders in this connected component, such that the connected component of the complement of $E_1 \cup E_2$ that does not contain the $C_i$ does not contain any pair of homologous horizontal cylinders. (That is, pick $E_1$ and $E_2$ to be ``innermost", or ``as close together as possible".) The above argument shows that the top of $E_1$ cannot be glued as in Figure \ref{F:C1C2} to the bottom of $E_2$ or vice versa, and hence we can find a cylinder $D_i$ in the connected component of the complement of $E_1 \cup E_2$ that does not contain the $C_i$. This $D_i$ cannot be homologous to any other horizontal cylinder. 
\end{proof}

\section{Proof of Theorem \ref{T:main}}\label{S:main}

We will prove Theorem \ref{T:main} by induction on genus and the dimension of the stratum. The base case is given by Lemma \ref{L:genus2}. 
 \ann{R1: Make clear that proofs depend on Theorem 1.5 which will be proven at end. \\ A: We added a remark saying that we are using Theorem 1.5 to the beginning of the previous section to clarify this. Since it is impossible to read Section 4 without looking at Section 3, we hope this suffices.}

That is, we will consider a full rank affine invariant submanifold $\cM$ of genus $g>2$ surfaces, and we will assume $\cM$ is not a connected component of a stratum. By induction, we assume the following.  

\begin{ass}\label{A0}
Theorem \ref{T:main} is true in genus at least 2 and strictly less than $g$, as well as in all strata of genus $g$ surfaces where the Abelian differentials have strictly fewer zeros than those in $\cM$.\ann{R2: Consider highlighting the inductive hypothesis as you did with Assumptions 1 and 2.\\A: We highlighted the inductive hypothesis as suggested.}
\end{ass}

 Our goal is to prove that $\cM$ is  a hyperelliptic locus. Throughout the proof, we will refer to the surface $(X,\omega)\in \cM$ produced by Proposition \ref{P:setup} and the notation given in Remark \ref{R1:notation}.\ann{A: Fixed broken reference to remark.} 
We will assume that $(X,\omega)$ has two additional properties. 

\begin{ass}\label{A1}
The relative cohomology class of $\omega$ does not sit in any proper subspace of $T_{(X,\omega)}\cM$ defined over $\overline{\bQ} \cap \bR$. 
\end{ass}

\begin{ass}\label{A2}
Each $D_i$ contains a vertical saddle connection. 
\end{ass}

The first assumption can be arranged by generically perturbing the real parts of the period coordinates of $(X,\omega)$, and the second assumption, which is made solely for concreteness of exposition, can be arranged by additionally performing cylinder deformations. \ann{R2:  ``$(X_i, \omega_i, S_i)$ is obtained vertically  squashing  each $D_i$."  Consider rephrasing this so as to emphasize that not all of the $D_i$ are being collapsed at once:  for each $i$, $(X_i, \omega_i, S_i)$ is obtained vertically  squashing  $D_i$.\\A: We made this clarification as suggested.}\ann{R2: This argument is not sufficiently justified. For a surface in $\cH(1,1)$ with three cylinders, it is possible to make everything line up just so, so as to collapse the large cylinder and get two homologous cylinders that were not previously homologous. ...\\A: We clarified this as suggested.}

Let $(X_i, \omega_i, S_i)$ denote the limit of $a_t^{D_i}(X,\omega)$ as $t\to -\infty$ in the partial compactification of strata discussed in \cite{MirWri}. This ``what you see is what you get" partial compactification can be described formally by taking limits in Deligne-Mumford with all zeros marked, and then deleting zero area components and separating nodes. It should be contrasted with larger compactifications, see for example \cite{Many} and the references therein. 

In other words, for each $i$, the surface $(X_i, \omega_i, S_i)$ is obtained by vertically ``squashing"  $D_i$, i.e. letting the height go to zero while keeping the circumference constant. Because we have assumed that each $D_i$ contains a vertical saddle connection, this results in the collision of zeros. 

Let $(X_i, \omega_i)$ be the same limit surface with marked points forgotten. On $(X_i, \omega_i, S_i)$ all the cylinders in $\cC$, as well as all $D_j, j\neq i$, persist and have the same area and circumference as on $(X, \omega)$. Note also that squashing $D_i$ does not create new horizontal cylinders, because squashing $D_i$ does not change the horizontal straight line flow outside of $D_i$. Hence, because of Assumption \ref{A1}, squashing $D_i$ does not create any new cylinders homologous to those in $\cC$.


\begin{figure}[h]
\includegraphics[width=0.5\linewidth]{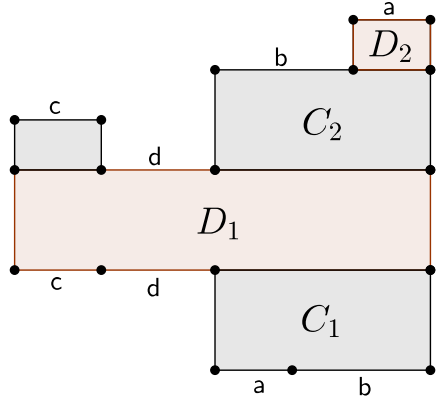}
\caption{In this example, if $D_1$ is squashed, the surface becomes disconnected. (One connected component is the  horizontal cylinder bounded above by $c$, and the other component  is $C_1\cup C_2 \cup D_2.$)  If either $D_1$ or $D_2$ are squashed then  $C_1$ and $C_2$ ``merge" to form a larger cylinder.}
\label{F:Disconnect}
\end{figure}

\begin{lem}\label{L:component}
On each $(X_i, \omega_i)$ there is a unique component that contains all the cylinders in $\cC$ as well as all $D_j, j\neq i$. 
\end{lem}

\begin{proof}
$(X, \omega)$ minus the  core curves of $C_i$ and $C_{i+1}$ has two connected components, one of which contains all $D_j, j\neq i$, as well as all $C_j, j\neq i, i+1$ and ``half" of the two cylinders $C_i$ and $C_{i+1}$ that were\ann{A: Typo found by R2 fixed.} cut in two by the removal of their core curves. Thus all the  cylinders in question other\ann{R1: what is domain that has two sides labeled $c$? Is that cylinder one of the $C_i$? It seems not. Is that region why the squashing $D_1$ disconnects? Please explain.\\A: We added some explanation to the caption.}
\ann{R1: why doesn't this figure contradict end of proof of 3.4. Namely above $C_1$ and below $C_2$ is $D_1$ and something else.\\A: We did not make any change in response to this comment (but we did explain in our response to report R1).}
\ann{R1: Why should $C_i$ and $C_{i+1}$ lie exactly over each other. \\A: We did not make any change in response to this comment (but we did explain in our response to report R1).} than $D_i$ are contained in a connected set disjoint from $D_i$, and so will remain in a connected set after squashing $D_i$.  
\end{proof}

Let $\Sigma$ denote the set of zeros of $\omega$, and let $V_i\subset H_1(X, \Sigma)$ denote the\ann{A: Typo found by R3 fixed.} kernel of the induced map on relative homology of the natural collapse map from $X$ to $X_i$ (see \cite[Section 9]{MirWri} for a formal definition).  Note  that $V_i$ always includes all vertical saddle connections in $D_i$.\ann{A: Typo found by R2 fixed.} (In Figure \ref{F:Disconnect}, if $D_1$ is squashed, $V_1$\ann{A: Typo found by R2 fixed.} additionally includes the saddle connection labelled $d$. Indeed, $(X_1,\omega_1)$ consists of one component corresponding to the horizontal cylinder below $c$, and one component corresponding to $C_1\cup C_2 \cup D_2$, glued together at a point, and the collapse map sends $d$ to this point. Note that the two components may be separated by ``ungluing" at the point, and the map remains well defined on relative cohomology.) \ann{R3: why is $d$ in kernel? It is not crushed\\ A: We added some explanation in the parentheses. See \cite[Section 9]{MirWri} for more details, which has been updated in November and is now more readable.}

The tangent space to the stratum of $(X_i, \omega_i, S_i)$ can be naturally identified with $\Ann(V_i)\subset H^1(X, \Sigma)$, the space of cohomology  classes that are zero on all relative homology classes in $V_i$ \cite[Section 9]{MirWri}. The space $\Ann(V_i)\cap T_{(X, \omega)}\cM$ can be thought of as the space of deformations of $(X, \omega)$ that stay in $\cM$ and do not change the part of the surface that is  crushed by squashing $D_i$. This space should be the tangent space to the component of the boundary of $\cM$ that contains $(X_i, \omega_i, S_i)$, and  this has been verified after projecting to each connected component \cite[Theorem 2.9]{MirWri}.\ann{A: Updated reference (issue noticed by R2).}  

\begin{lem}\label{L:boundary}
Let $\pi_i$ be the map that forgets all but the component of $(X_i, \omega_i)$ given by Lemma \ref{L:component}. Then $\pi_i(X_i, \omega_i, S_i)$ is contained in an affine invariant submanifold $\cM_i^*$ whose tangent space at $\pi_i(X_i, \omega_i, S_i)$ is given by $(\pi_i)_*(\Ann(V_i)\cap T_{(X,\omega)}\cM)$.\ann{R2: If $\pi_i$ is a map at the level of surfaces and $\Ann(V_i) \cap T_{(X,\omega)}\cM$ lies in cohomology, shouldn't $\pi_i$ act by pull-back $(\pi_i)^*$ and not push-forward? I realize that as written, it exactly agrees with [MW], but I found this point confusing.\\A: Added a definition of $(\pi_i)_*$. We apologize that the notation may be confusing; it is now too late to change it in [MW], and at least in this paper we'd like to have consistent notaiton with [MW].}

$\pi_i(X_i, \omega_i)$ is contained in an affine invariant submanifold $\cM_i$ whose tangent space at $\pi_i(X_i, \omega_i)$ is given by the image of $T_{\pi_i(X_i,\omega_i, S_i)}\cM_i^*$ under the forgetful map. 
\end{lem}

If $\Sigma_i$ is the set of zeros of $\omega_i$, the forgetful map is the restriction map from linear functionals on $H_1(X, \Sigma_i \cup S_i)$ to\ann{A: Typo found by R2 fixed.} linear functionals on $H_1(X, \Sigma_i)\subset H_1(X, \Sigma_i \cup S_i).$ One may define $(\pi_i)_*: \Ann(V_i) \to H^1(X_i, S_i)$ to be the inverse of the isomorphism $\pi_i^*: H^1(X_i, S_i)\to \Ann(V_i)\subset H^1(X, \Sigma)$.

\begin{proof}
The first statement is an application of \cite[Theorem 2.9]{MirWri}\ann{A: Updated reference.}. This second statement is a corollary of the first by \cite[Lemma 6]{LMW}.  
\end{proof}

\begin{lem}\label{L:stillfullrank}
Each $\cM_i$ has full rank. 
\end{lem}

\begin{figure}[h!]
\includegraphics[width=\linewidth]{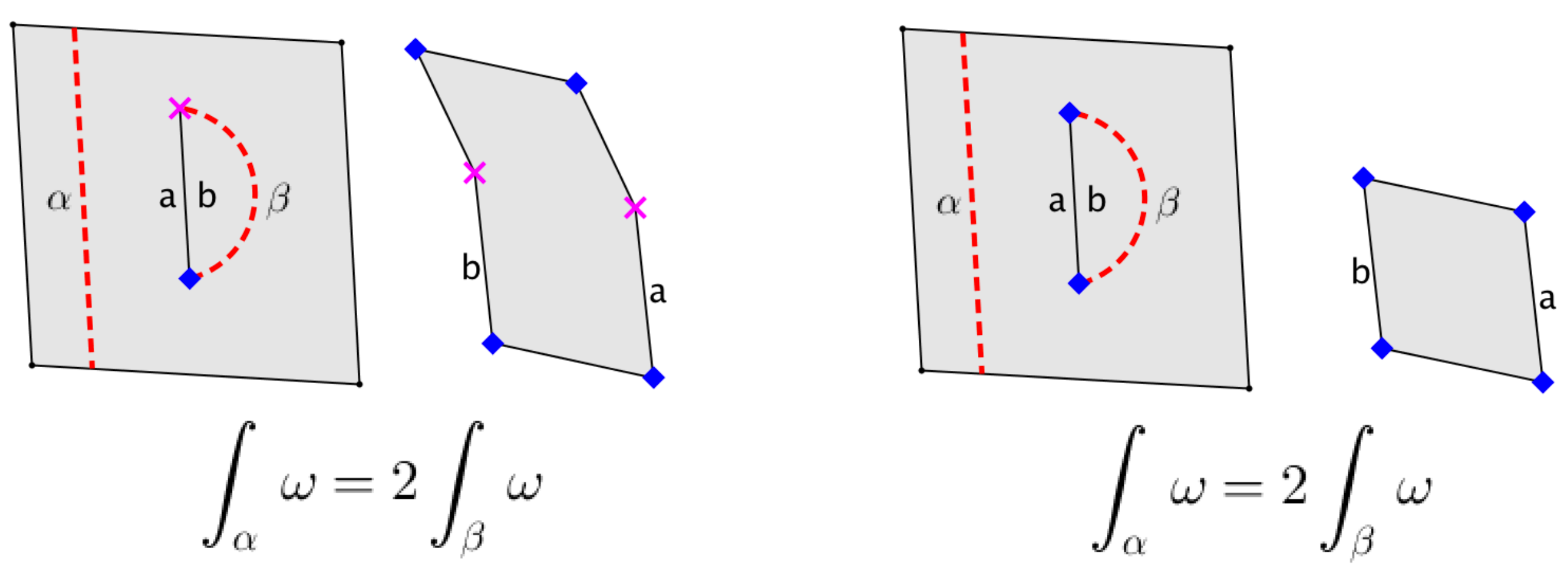}
\caption{Left: A surface in $\cH(1,1)$ with an equation on relative periods. Right: A degeneration of the surface to $\cH(2)$.} \label{F:RelToAbs}
\end{figure}

\begin{rem}
The  cautionary example illustrated in Figure \ref{F:RelToAbs} shows that outside of the context of affine invariant submanifolds, linear equations on periods that do not restrict absolute periods can in fact restrict absolute periods after a degeneration.
\end{rem}

\begin{figure}[h]
\includegraphics[width=0.5\linewidth]{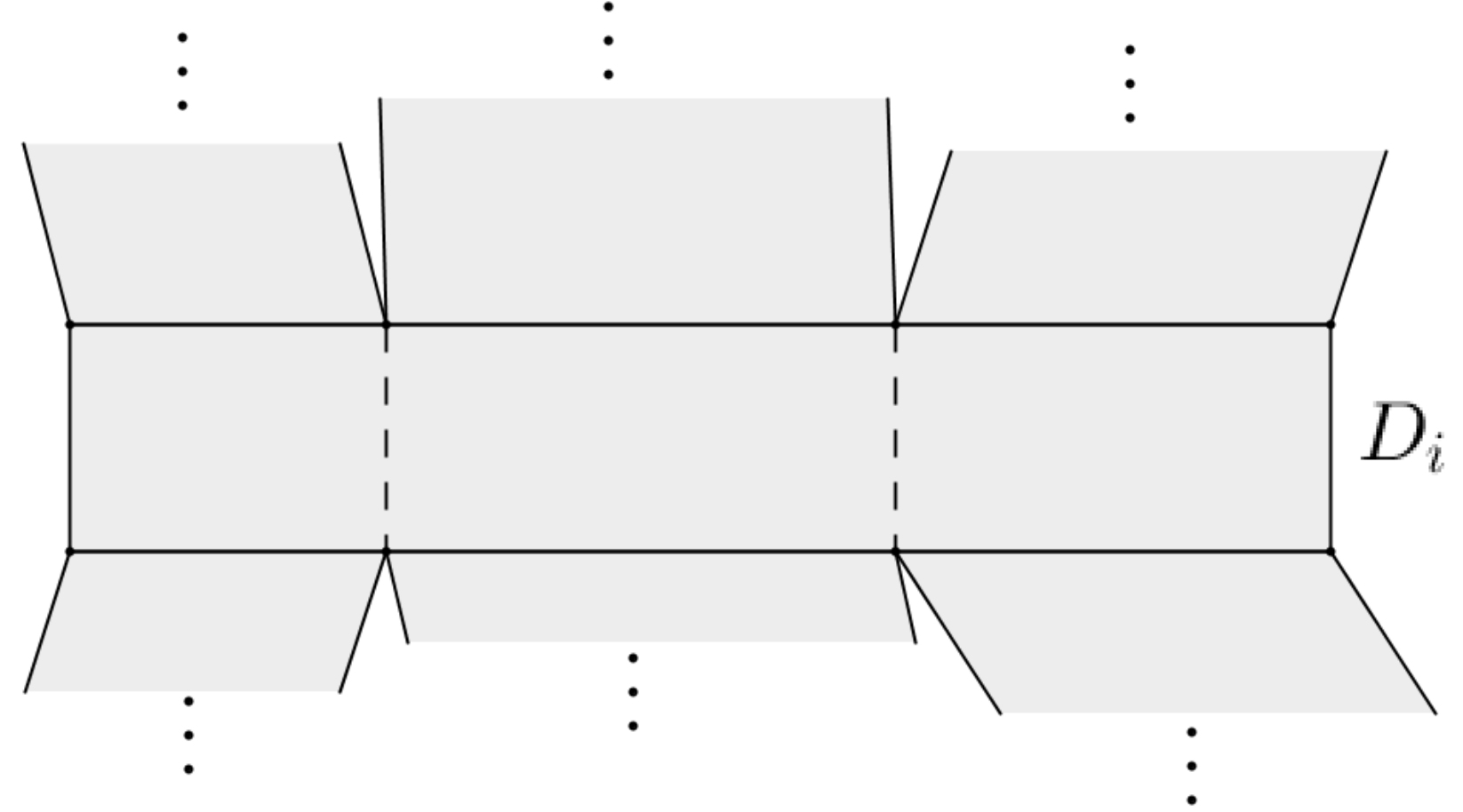}
\caption{This figure illustrates one situation that is avoided by Assumption \ref{A1}. In this highly ungeneric picture, the cylinders above $D_i$ have exactly the same circumferences as those below $D_i$. If $D_i$ is vertically squashed, the genus of the degeneration is 2 less than that of the original surface.} 
\label{F:Ungeneric}
\end{figure}

\begin{proof}
By  the second statement of  Proposition \ref{P:setup},  there are  horizontal cylinders $F_1, \ldots, F_{g-1}$  so that $F_1, \ldots, F_{g-1}, D_i$ are disjoint and have core curves that are linearly independent in homology. Let $\cF_j$ denote the set of all cylinders homologous to $F_j$. (The example in Figure \ref{F:Disconnect} shows that it is not always possible to assume  $\cF_j=\{F_j\}$.)

We consider the twists $u^{\cF_j}, j=1,\ldots, g-1$ and their projections via the map $\pi_i:X\to X_i$. The $u^{\cF_j}$ are symplectically orthogonal, and remain so after applying $\pi_i$, because they are Poincare dual to disjoint curves. Furthermore, the core curves of the $F_j$\ann{R2: ``the core curve of the $F_i$ are homologically".  Should this be  ``the core curve of the $F_j$ are homologically"? ($j$ vs $i$)\\A: Yes; we corrected corrected this, at several places in the proof.} are homologically independent on $(X_i, \omega_i)$ since by Assumption \ref{A1} their periods are  linearly independent over $\bQ$. Hence the $\pi_i(u^{\cF_j}), j=1, \ldots, g-1$ give an isotropic subspace of dimension $g-1$ in the absolute cohomology of $(X_i, \omega_i)$.
   
If $(X_i,\omega_i)$ has genus at most $g-1$, then this gives a Lagrangian subspace, and so $\cM_i$ has full rank.

So suppose the genus of $(X_i, \omega_i)$ is $g$. Hence no zero appears on both the top and bottom of $D_i$, 
and $\Ann(V_i)$ is spanned by the vertical saddle connections in $D_i$.

Because $\cM$ has full rank, there is a real cohomology class $v\in T_{(X,\omega)}\cM$ that is nonzero on the core curve of $D_i$ and that is zero on the core curves of the $F_j, j=1, \ldots, g-1$. By Assumption \ref{A1}\ann{A: Typo found by R2 fixed.}, the vertical saddle connections\ann{R2:  ``connections in $D$ must all give rise".  Please change to  ``connection in $D_i$ must all give rise". ($D$ vs $D_i$.)\\A: Fixed.} in $D_i$ must all give rise to parallel saddle connections on the surface obtained by adding a small real multiple of $v$ to the period coordinates of $(X,\omega)$. Hence we may add to $v$ a multiple of $u^{D_i}$ to assume  that $v$ is zero on all vertical saddle connections in $D_i$, and conclude that   $v\in \Ann(V_i)\cap T_{(X,\omega)}\cM$.

To show that $\cM_i$ has full rank, it now suffices to show that  the twists in $\cF_j$ together with $v$ spans a Lagrangian subspace.  

Let $\alpha_j, j=1, \ldots, g$ be the core curves of $F_1, \ldots, F_{g-1}, D_i$, and pick $\beta_j$ so $\{\alpha_j, \beta_j\}$ is a symplectic basis of $H_1(X, \bZ)$. The symplectic pairing on cohomology classes $v,w \in H^1(X, \bR)$ is given by 
$$ \sum_{j=1}^g v(\alpha_j)w(\beta_j)- v(\beta_j)w(\alpha_j).$$
The twist in $\cF_j$ is nonzero  on $\beta_j$ and zero on all other elements of the basis. The  deformation $v$ is zero on all $\alpha_j, j<g$, so $v$ is symplectically orthogonal to the twist in each $\cF_j$. 
\end{proof}

\begin{lem}
Each $\cM_i$ is contained in a stratum of surfaces of genus at least 2. 
\end{lem}

\begin{proof}
 $D_{i+1}$ and a cylinder from $\cC$ are disjoint on $\pi_i(X_i,\omega_i)$, and their core curves  have different holonomy by Assumption \ref{A1}. 
\end{proof}

\begin{lem}\label{L:Mihyp}
Each $\cM_i$ is a hyperelliptic locus or a hyperelliptic connected component of a stratum. 
\end{lem}

\begin{proof}
By the inductive hypothesis and Lemma \ref{L:stillfullrank}, it suffices to show that each $\cM_i$ is not a connected component of a non-hyperelliptic stratum.  

A cylinder on an unmarked surface that contains a marked point in its interior is considered to be two cylinders on the marked surface. Hence the cylinders of $\cC$ remain well defined on $(X_i, \omega_i, S_i)$.  However it is possible that some cylinders in $\cC$ ``merge" on $(X_i, \omega_i)$, as in Figure \ref{F:Disconnect}, because the marked points  are forgotten. 

If none of the cylinders in $\cC$ merge on $(X_i, \omega_i)$,  then because their moduli must remain related,  $\cM_i$ cannot be a stratum. 

If to the contrary $C_i$ and $C_{i+1}$ merge to create a single larger cylinder on $(X_i, \omega_i)$, then there must be marked points separating $C_i$ and $C_{i+1}$ on $(X_i, \omega_i, S_i)$.  

The heights of the cylinders in $\cC$ on $(X_i, \omega_i, S_i)$ satisfy some linear equation.  Let $\gamma_i$ be the core curve of $C_i$. Consider the relative cohomology class Poincare dual to $\gamma_i-\gamma_{i+1}$, which is a multiple of the derivative of a deformation which increases the height of $C_i$ and decreases the height of $C_{i+1}$. We now consider two cases. 

In the first case, suppose the dual of $\gamma_i-\gamma_{i+1}$ is in $T\cM_i^*$. In this case the equation on the heights of the cylinders in $\cC$ on $(X_i, \omega_i, S_i)$ depends only on the sum of the heights of $C_{i}$ and $C_{i+1}$ and not on the individual heights of $C_i$ and $C_{i+1}$, so when the marked points are forgotten there is still a linear equation that holds on the heights. This shows $\cM_i$ is not connected component of a stratum. 

In the second case, suppose the dual of $\gamma_i-\gamma_{i+1}$ is not in  $T\cM_i^*$. This implies that the marked points separating $C_i$ from $C_{i+1}$ are not all unconstrained, since if they were all unconstrained then moving them all upwards at a constant rate would realize the deformation dual to $\gamma_i-\gamma_{i+1}$. So by Theorem \ref{T:apisa}  we conclude that $\cM_i$ is not a non-hyperelliptic connected component of a stratum. 
\end{proof}

\begin{lem}\label{L:hypinv}
If  a translation surface $(Y,\eta)$ is hyperelliptic, then each cylinder is homologous to at most one other. Homologous pairs of cylinders are exchanged by the hyperelliptic involution. 
\end{lem}


\begin{proof}
Say $\cC$ is a collection of homologous cylinders. Let $\iota$ be the hyperelliptic involution. Because $\iota$ acts as $-1$ on homology, each of the components of $(Y,\eta)$ minus the core curves of $\cC$ must be preserved, and $\iota$ must interchange the two boundary components. 
\end{proof}

\begin{lem}\label{L:L2}
The number $k$ of cylinders  in $\cC$ is  2. 
\end{lem}

\begin{proof}
Lemmas \ref{L:Mihyp} and \ref{L:hypinv} imply that $(X_i, \omega_i)$ has at most two homologous cylinders. Since at most two cylinders (namely, $C_i$ and $C_{i+1}$) can merge on $(X_i, \omega_i)$, this shows $k\leq 3$. Furthermore, for this same reason exactly two cylinders must always merge on $(X_i, \omega_i)$ for each $i$. 

Suppose $k=3$, and suppose without loss of generality that  the area of $C_1\cup C_2$ is strictly larger than the area of $C_3$. On $(X_i, \omega_i)$, the cylinders $C_1$ and $C_2$ merge to give a single cylinder homologous to $C_3$ but of larger area. This contradicts Lemmas \ref{L:Mihyp} and \ref{L:hypinv}, so we conclude that $k=2$. \ann{R2: Proof of Lemma 4.10, Second Paragraph: This is too brief because it gives the impression that there are two assumptions in this paragraph and they cannot be made simultaneously: 1) $A(C_1 \cup C_2) > A(C_3)$, 2) $C_1$ and $C_2$ merge. Consider adding the following sentence at the end of the  first paragraph to justify that it is always possible to renumber so that the area assumption holds: Furthermore, for this same reason exactly two cylinders must always merge on $(X_i, \omega_i)$.\\A: Clarified as suggested.}
\end{proof}

\begin{lem}\label{L:nofreepoints}
The set of marked points on every surface in $\cM_i^*$, $i=1,2$, is\ann{A: Typo found by R2 fixed.} preserved by the hyperelliptic involution.
\end{lem}

\begin{proof}
By Theorem \ref{T:marked}, it suffices to show that $\cM_i^*$ does not have unconstrained marked points. 

Suppose to the contrary that  $\cM_1^*$ has unconstrained marked points. Then so must the boundary affine invariant submanifold obtained by squashing $D_2$, forgetting marked points, and then squashing $D_1$ and remembering marked points (and in each case applying the projection to the component containing the cylinders in $\cC$). However, since $\cM_2$ is hyperelliptic, the second step cannot produce unconstrained marked points, because the set of marked points must be invariant under the hyperelliptic involution. 

Squashing $D_1$ and remembering marked points, and then squashing $D_2$ and forgetting new marked points, gives the same affine invariant submanifold as doing it in the other order. This is true since the boundary affine invariant submanifolds are locally the set of limits of degenerations \cite[Sections 7, 9]{MirWri},\ann{R2:  It seems that Section 9 is the more appropriate reference.\\A: Added Section 9 to the reference.} and in both cases\ann{A: Typo found by R2 fixed.} the set of limits is equal to the result of squashing both $D_1$ and $D_2$, projecting to a connected component, and forgetting the marked points on the $D_2$ side of $C_1\cup C_2$.
\end{proof}

\begin{lem}\label{L:subsetofhyp}
$\cM$ is contained in the hyperelliptic locus. 
\end{lem}\ann{R2: Figure 4.4: Why include the left side of Figure 4.4? It is not clear to me why it is relevant to Remark 4.13.\\ A: It illustrates how we discovered the example on the right side, and may help some readers to understand it.}

\begin{figure}[h]
\includegraphics[width=\linewidth]{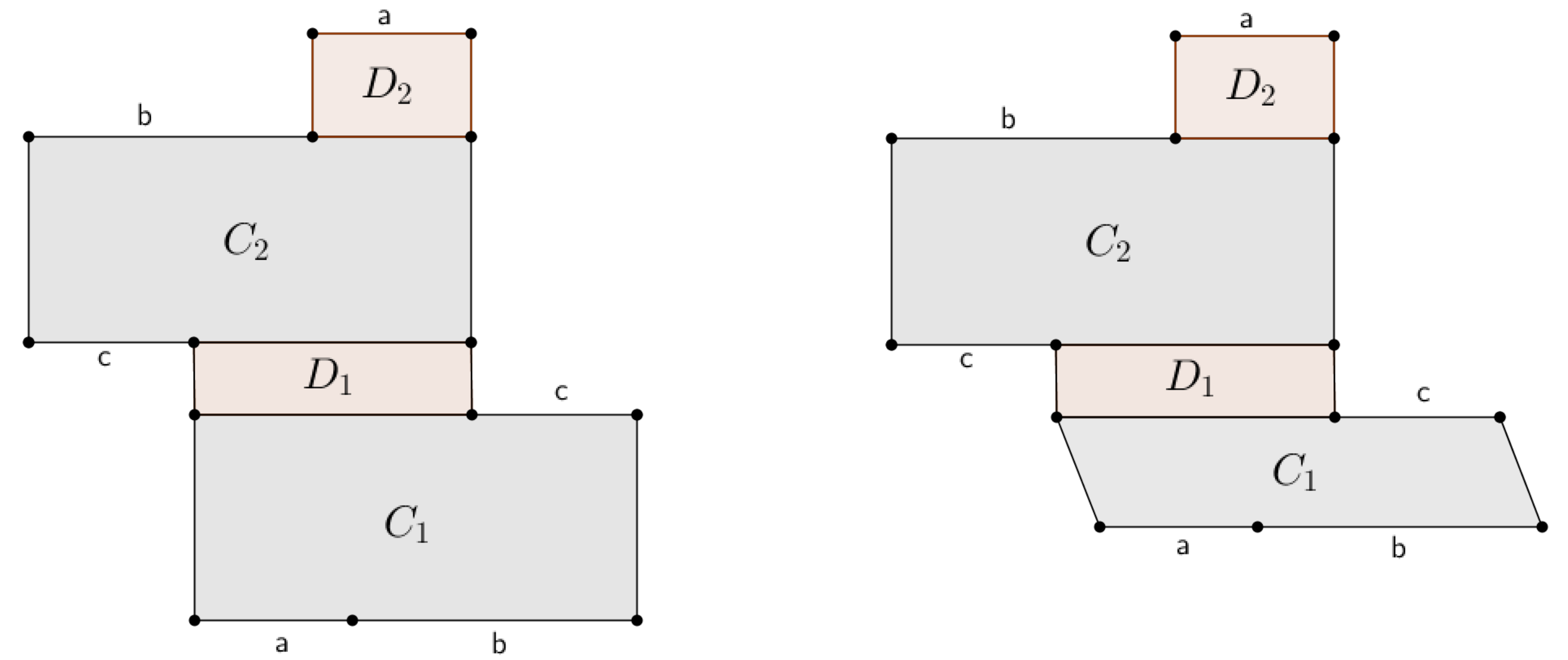}
\caption{Left: A hyperelliptic surface. Right: deforming $C_1$ gives a surface without an involution.}
\label{F:AlmostHyp}
\end{figure}

\begin{rem}
The cautionary example illustrated on the right of Figure \ref{F:AlmostHyp} shows that outside of the context of affine invariant submanifolds, a non-hyperelliptic surface with cylinders $D_1, D_2$ may become hyperelliptic after either $D_i$ is squashed and marked points are forgotten. 
\end{rem}

\ann{R2:  ``Thus, $C_1 \cup C_2$ has at most one involution...  This assertion does not seem like a completely obvious consequence of Assumption 4.1. As a more concrete alternative, why not use the same argument from the previous proof. 
\\A: We added some explanation to make it clear how it follows from Assumption 4.1. (What you suggested also works, but we prefer to minimize changes to the proofs at this point.)}
\begin{proof}
By Lemmas \ref{L:Mihyp} and \ref{L:nofreepoints},  both $\pi_1(X_1, \omega_1, S_1)$ and $\pi_2(X_2, \omega_2, S_2)$ have involutions $\tau_i$ respecting marked points. 

For $i=1,2$, let $R_i$ be $\overline{C_1\cup C_2}$, union the connected component of $X\setminus (C_1\cup C_2)$ that contains $D_{i+1}$.  $\tau_i$ induces an involution on $R_i$ that preserves zeros of $\omega$. 

By  Assumption \ref{A1},  the saddle connections on the bottom of $C_2$ have distinct lengths. Since any involution of $\overline{C_1 \cup C_2}$ must send each saddle connection on the top of $C_1$ to a saddle connection on the bottom of $C_2$ of the same length, and since there is a unique such saddle connection on the bottom of $C_2$, $\overline{C_1 \cup C_2}$ has at most one involution exchanging $C_1$ and $C_2$ and preserving the set of zeros of $\omega$. Hence $\tau_1|_{R_1}$ and $\tau_2|_{R_2}$ agree on $R_1\cap R_2=\overline{C_1 \cup C_2}$, and we can glue these involutions to get an involution $\tau$ on $(X,\omega)$.  

 $\tau$ is a hyperelliptic  involution since $X/\tau$ has genus 0. Indeed, $X/\tau$ can be obtained  by gluing together the two genus 0 surfaces obtained as the quotient of $R_i$ by $\tau_i|_{R_i}$.
  
 By \cite[Corollary 1.3]{Wfield} and Assumption \ref{A1}, $(X,\omega)$ has dense orbit in $\cM$. Hence every surface in $\cM$ is hyperelliptic.  
\end{proof}

\begin{lem}
$\cM$ is equal to the hyperelliptic locus in the ambient stratum. 
\end{lem}

\begin{proof}
It remains only to show that $\cM$ is not properly contained in the hyperelliptic locus. Let $v\in H^1(X, \Sigma)$  be supported on one component of $X$ minus the core curves of $C_1$ and $C_2$ and suppose $v$ is contained in the  -1\ann{Typo found by R2 fixed.} eigenspace for the hyperelliptic involution. By Lemma \ref{L:Mihyp} and Lemma \ref{L:boundary}, $v \in T\cM$.
\ann{R2: Proof of Lemma 4.14, Paragraph 2:  there is an affine map with derivative 1 taking $\overline{C_1 \cup C_2}$ on $g(X, \omega)$ to $\overline{C_1 \cup C_2}$ on $(X', \omega')$ and that maps zeros to zeros.  This is certainly true for the open sets $C_1 \cup C_2$. However, it is not clear that this extends to the boundaries. Consider clarifying this. ... \\A: Added explanation in parentheses.}

Consider a generic hyperelliptic surface $(X,\omega)\in \cM$\ann{R2:   ``surface $(X, \omega)$ with". Consider changing to  ``surface $(X, \omega) \in \cM$ with"  for emphasis because the term generic might be misleading here.\\A: Done.} with a pair of homologous cylinders $C_1$ and $C_2$, and let $(X', \omega')$ be any nearby hyperelliptic surface. Then $(X', \omega')$ can be obtained from $(X,\omega)$ as follows. First, apply $g\in GL(2,\bR)$ to get $g(X,\omega)$, where $g$ is chosen so there is an affine map with derivative 1 taking $\overline{C_1\cup C_2}$ on $g(X,\omega)$ to $\overline{C_1\cup C_2}$ on $(X', \omega')$ and that maps zeros to zeros. (Note that since $(X,\omega)$ is generic, $C_1$ and $C_2$ are simple cylinders.) Then,  apply a deformation supported on one side of $g(X,\omega)\setminus (C_1\cup C_2)$, and then a deformation supported on the other side, to obtain $(X', \omega')$. 

This shows that the deformations from the first paragraph, together with $\Re(\omega)$ and $\Im(\omega)$, span the tangent space to the hyperelliptic locus.  
\end{proof}

This concludes the proof of Theorem \ref{T:main}.

\section{Proof of Theorem \ref{T:free}}\label{S:free}

A digraph is a pair $\Gamma=(V,E)$, where $V$ is a finite set of vertices, and $E$ is a finite multi-set of ordered pairs $(x,y)\in V\times V$. An element $(x,y)\in E$ is called a directed edge from $x$ to $y$. We say that $x$ is the tail of $(x,y)$ and that $y$ is the tip. A directed graph is allowed to have multiple edges between two vertices or from a vertex to itself.

A directed path in a digraph is a nonempty sequence of edges such that the tip of each is the tail of the next. The path is said to go from the tail of the first edge to the tip of the last. A directed loop is a directed path from a point to itself. 

A digraph $\Gamma(V,E)$ is called strongly connected if, for all $x,y\in V$, there is a directed path from $x$ to $y$ and a directed path from $y$ to $x$.  

A directed loop is said to be embedded if it passes through each edge and each vertex at most once. Let $\bQ^E$ denote the vector space with basis $E$. Each directed loop gives an element of $\bQ^E$ by taking the sum of edges that appear in the loop. Define the loop space $L(\Gamma)\subset \bQ^E$ to be the span of all directed embedded loops in  $\bQ^E$.

\begin{lem}\label{L:digraph}
If $\Gamma$ is strongly connected, then $\dim L(\Gamma)=|E|-|V|+1$. 
\end{lem}

Compare to \cite[Theorem 10.1.4]{digraphs}; since we haven't been able to find a precise reference, we provide a proof. 

\ann{R2: For undirected graphs, this is a well-known result readily available on Wikipedia. See
circuit rank and cycle space.
However, it is admittedly much more challenging to  find a reference for strongly
connected digraphs. Consider adding the following:
Digraphs: Theory, Algorithms and Applications by Jorgen Bang-Jensen and
Gregory Gutin, pg. 547 Theorems 10.1.3 and 10.1.4.\\ A: Reference  added.}

\begin{proof}
We induct on $|V|$. If $|V|=1$, the result is obvious. So assume $|V|>1$. 

Let $\ell$ be an embedded directed loop in $\Gamma$, and let $\Gamma'=(V', E')$ be the digraph resulting from collapsing $\ell$ to a point. One can check that $\Gamma'$ is strongly connected. 

\begin{figure}[h]
\includegraphics[width=0.7\linewidth]{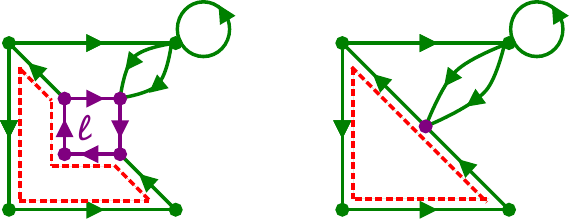}
\caption{Left: A digraph $\Gamma$. Right: Contracting a loop $\ell$ in $\Gamma$ gives $\Gamma'$. The dashed line on the right is an embedded directed loop; and on the left, its  lift.}
\label{F:CylDigraph}
\end{figure}

An embedded loop in $\Gamma$ does not necessarily map to an embedded loop in $\Gamma'$, because it may pass through a vertex more than once. However, the image of an embedded loop will  map to a sum of embedded loops in  $L(\Gamma')$. Hence we get a map $L(\Gamma)\to L(\Gamma')$.

Each embedded loop in $\Gamma'$ lifts to an embedded directed loop in $\Gamma$, so in fact there is a short exact sequence 
$$0\to \bQ \cdot \ell \to L(\Gamma)\to L(\Gamma')\to 0.$$

If $\ell$ has $k$ edges, then $|V'|=|V|-k+1$ and $|E'|=|E|-k$, and so the result follows by induction.  
\end{proof}  
  
Let $(X,\omega)$ be a horizontally periodic translation surface. We define its cylinder digraph to be the digraph with one vertex for each horizontal cylinder, and one directed edge for each horizontal saddle connection. The directed edge associated to a horizontal saddle connection goes from the cylinder below the saddle connection to the cylinder above it.  The cylinder digraph is strongly connected. 

\begin{figure}[h]
\includegraphics[width=0.7\linewidth]{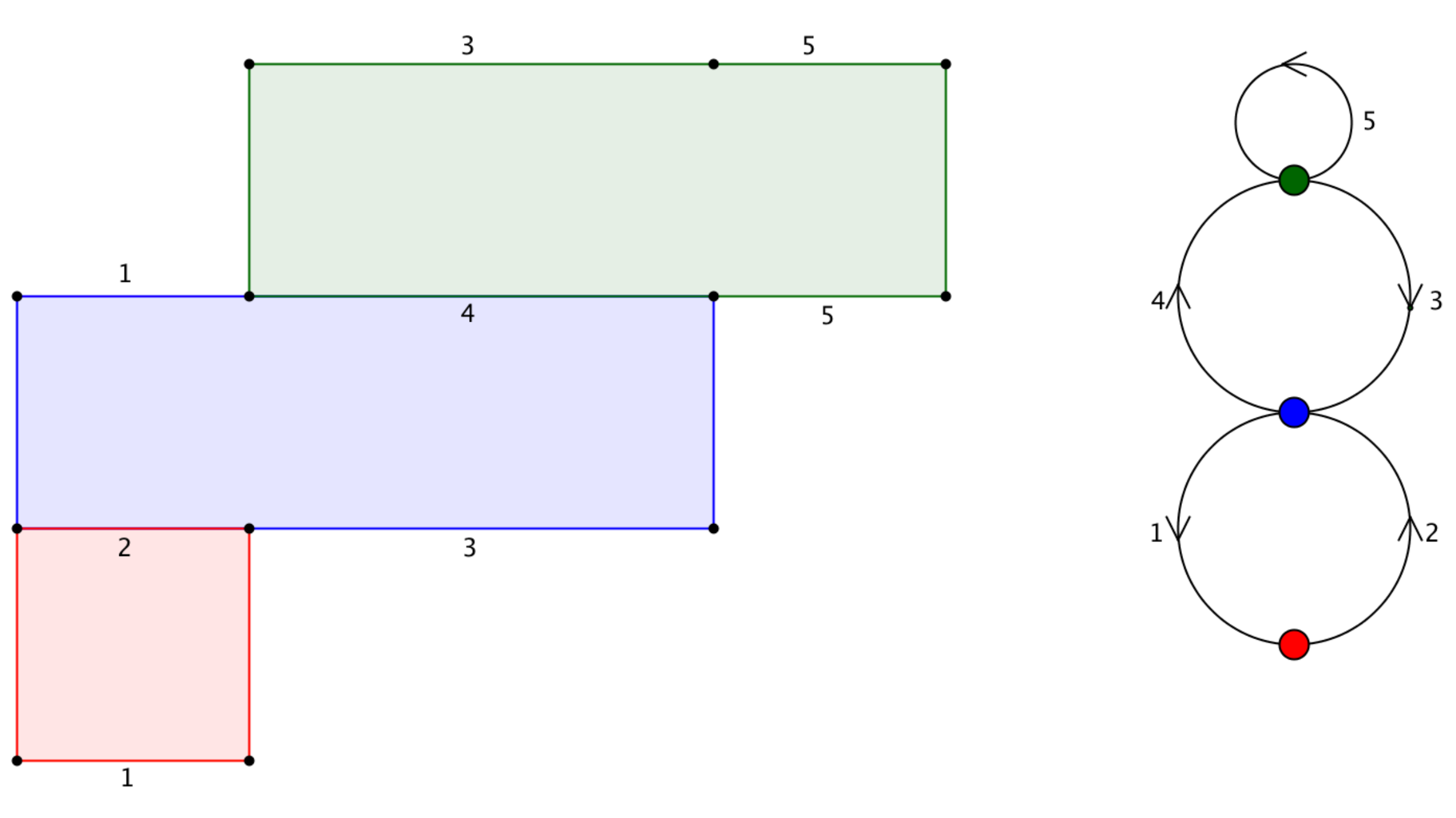}
\caption{A horizontally periodic surface and its cylinder digraph. The colors and labels are for illustrational purposes only.}
\label{F:CylDigraph}
\end{figure}


Given a horizontally periodic surface $(X,\omega)\in \cM$,  the cylinder preserving space is defined in \cite{Wcyl} to be the subspace of $T\cM$ that is zero on core curves of cylinders, and the twist space of $(X,\omega)$ is the subspace of $T\cM$ that is zero on all horizontal saddle connections. 
The following summarizes \cite[Lemma 8.6, Corollary 8.11]{Wcyl}. 

\begin{lem}
Let $(X,\omega)\in \cM$ be horizontally periodic. The cylinder preserving space  has codimension at most $\rank \cM $ in $T\cM$. 
If the twist space is not equal to the cylinder preserving space, then there is a horizontally periodic surface in $\cM$ with more horizontal cylinders than $(X,\omega)$. 
\end{lem}

\begin{lem}\label{L:CPTwist}
Let $\cM$ be an affine invariant submanifold as in Theorem \ref{T:free}. Let $(X,\omega)\in \cM$ be horizontally periodic in $\cM$, and let $\Gamma$ be the cylinder digraph. If the cylinder preserving space is equal to the twist space, then $\dim L(\Gamma)  \leq \rank \cM$.
\end{lem}

\begin{proof}
For each embedded simple loop $\ell$ in the cylinder digraph, we may twist the horizontal cylinders of $(X,\omega)$ to obtain a surface with a vertical cylinder that passes through exactly the saddle connections and cylinders of $\ell$. Stretching this vertical cylinder gives an element of $T\cM$ which is equal to a nonzero constant on the horizontal saddle connections of $\ell$, and is zero on on all other horizontal saddle connections. The cohomology classes obtained in this way from different loops $\ell$ are linearly independent, so their span $L\subset T\cM$ has dimension $\dim L(\Gamma)$. 

Since the cylinder preserving space has codimension at most $\rank \cM$, any subspace of $T\cM$ of dimension strictly greater than $\rank \cM$ must have non-trivial intersection with it. Thus, if $\dim L(\Gamma)>\rank \cM$, some nonzero element of $L$ is contained in the cylinder preserving space. 

Since every nonzero element of $L$ is nonzero on some horizontal saddle connection, no nonzero element of $L$ is in the twist space. We conclude that if $\dim L(\Gamma)>\rank \cM$ the cylinder preserving space is strictly larger than the twist space. 
\end{proof}

\begin{lem}\label{L:freetostrat}
Suppose $\cM$ is an affine invariant submanifold of genus $g$ Abelian differentials with $s$ distinct zeros. If $\cM$ contains a surface with $g+s-1$ horizontal free cylinders, then $\cM$ is a connected component of a stratum. 
\end{lem}

\begin{proof}
The twists in the cylinders span a subspace of $T\cM$ of dimension $g+s-1$ whose image under $p$ is isotropic.\ann{R3: What does isotropic mean here?\\ A: As usual, it means that the symplectic pairing of any two vectors in the subspace is zero.} Since $\ker(p)$ has dimension $s-1$, the image under $p$ has dimension $g$. Hence $\cM$ has full rank and $p(T\cM)=H^1(X)$. 

The intersection of the span of these twists with $\ker(p)$ has dimension $s-1$, which is the dimension of $\ker(p)$.  Hence $\ker(p)\subset T\cM$. We conclude that $T\cM=H^1(X,\Sigma)$.  
\end{proof}

\ann{R3: Proof of Theorem 1.5. You form digraph. You need horizontally periodic to do that. Why does hypothesis of Theorem 1.5 assume that?\\A:  Theorem 1.5 doesn't assume anything about horizontally periodic surfaces (because the required horizontally periodic surfaces are produced in the proof).}
\begin{proof}[Proof of Theorem \ref{T:free}]
Consider a surface $(X,\omega)\in \cM$ with the maximum number of horizontal cylinders. We consider the cylinder digraph $(V,E)$. Suppose $X$ has genus $g$ and $\omega$ has $s$ distinct zeros, so $|E|=2g-2+s$. 

Since every cylinder is free, the dimension of the twist space is equal to $|V|$. By Lemma \ref{L:digraph}, $|V|= |E|-\dim L(\Gamma)+1$. By Lemma \ref{L:CPTwist},  $|V|\geq  |E|-g+1=g-1+s$, since rank is always at most genus. The result now follows from Lemma \ref{L:freetostrat}.
\end{proof}

\section{Unfoldings of polygons}\label{S:billiards}

In this section, we give some foundational results on affine invariant submanifolds associated to rational polygons and prove Theorem \ref{T:billiards}.

Fix  numbers $\theta_1, \ldots, \theta_n\in (0, \pi)\cup (\pi, 2\pi)$  such that $\theta_i/\pi\in\bQ$  for each $i$ and $\sum_{i=1}^n\theta_i= (n-2)\pi$. Let $\bm\theta=(\theta_1, \ldots, \theta_n)$. Define $k$ to be the least common denominator of $\theta_i/\pi, i=1, \ldots, n$, and set $q_i=k\theta_i/\pi$. For each tuple $\mathbf{z}=(z_1, \ldots, z_n)\in \bC^n$ of distinct complex numbers, define $X_{\mathbf{z}}$ to be the normalization of the plane algebraic curve 
$$y^k=\prod_{i=1}^n (z-z_i)^{q_i}.$$
This Riemann surface  has an automorphism $T$ given by $T(y,z)=(\xi y, z)$, where $\xi=\exp(2 \pi I/k)$.\ann{R3: $2\pi I/k$ should be $2\pi i/k$ \\A: We will try to change all the square roots of $-1$ to upper case $I$ to be consistent. (We don't want to use lowercase $i$, because we use that frequently as an index.)} This automorphism generates the deck\ann{R4: Why ``deck" is written with capital ``D''?\\A: We changed it to lowercase throughout the paper.} group for the  covering map $(y,z)\mapsto z$. We will refer to $X_{\mathbf{z}}$ as a cyclic cover of $\bP^1$. 
Let $\Sigma_\mathbf{z}\subset X_{\mathbf{z}}$ be the union of the preimages of each $z_i$ such that $q_i$ does not divide $k$. 

For $\ell\in \bZ/k=\{0, 1, \ldots, k-1\}$, define $H^{1,0}_\ell(X_{\mathbf{z}}), H^{1}_\ell(X_{\mathbf{z}})$ and $H^{1}_\ell(X_{\mathbf{z}}, \Sigma_\mathbf{z})$ to be the $\xi^\ell$-eigenspaces for the induced action of $T$ on $H^{1,0}(X_{\mathbf{z}}), H^{1}(X_{\mathbf{z}})$ and $H^{1}(X_{\mathbf{z}}, \Sigma_\mathbf{z})$ respectively. 

\ann{R2: Lemma 6.1: Consider adding a reference for context. For example, Lemma 2.6 from the second author's paper on Schwarz triangle mappings proves this when $n = 4$.}\ann{R3: Does Lemma 6.1 need a proof?}\ann{R4: To my mind, Lemma 6.1 merits a reference. I was once told that it is proved in various versions in [four sources]
\\A: We added a reference. One problem is that there are so many possible references that it is hard to pick one; if any of you wish to suggest a single additional reference that is distinguished in some way we would be happy to include it in addition to or instead of the reference we gave.
}

Let $\{x\}\in [0,1)$ denote the fractional part of $x\in \bR$, and let $\lfloor x \rfloor=x-\{x\}$ denote the greatest integer less than or equal to $x$. For $\ell\in \bZ/k$, define $$t_i(\ell)=\left\{ \frac{\ell q_i}k\right\}\quad\text{\and}\quad t(\ell)=\sum_{i=1}^n \left\{ \frac{\ell q_i}k\right\}.$$

\begin{lem}\label{L:eigendim}
$H^{1,0}_{-\ell}(X_{\mathbf{z}})$ is spanned by $$   p(z)\prod_{i=1}^n (z-z_i)^{-t_i(\ell)} dz $$ 
where $p(z)$ is a polynomial of degree at most $t(\ell)-2$. In particular, $H^{1,0}_{\ell}(X_{\mathbf{z}})$ has dimension $t(-\ell)-1$.
\end{lem}

This lemma is standard and can be found in many sources. One reference that uses the same notation is  \cite[Lemma 2.6]{W1} (for the case $n=4$, which is identical to the general case). 

Here $\prod_{i=1}^n (z-z_i)^{-t_i(\ell)}$ is  shorthand for $y^{-\ell} \prod_{i=1}^n (z-z_i)^{\lfloor \frac{\ell q_i}k \rfloor}$.  
Since $H^{1}_\ell(X_{\mathbf{z}})$ is the direct sum of $H^{1,0}_\ell(X_{\mathbf{z}})$ and the complex conjugate of $H^{1,0}_{-\ell}(X_{\mathbf{z}})$, we also have the following. 

\begin{cor}\label{C:dim}
$H^{1}_\ell(X_{\mathbf{z}})$ has dimension $t(\ell)+t(-\ell)-2$.
\end{cor}

Lemma \ref{L:eigendim} gives in particular  that $$\omega_{\mathbf{z}}=\prod_{i=1}^n (z-z_i)^{\frac{\theta_i}{\pi}-1} dz\in H^{1,0}_1(X_\mathbf{z}). $$  

Let $\cM_{cyc}(\bm\theta)$ denote the set of all $(X_{\mathbf{z}}, \omega_{\mathbf{z}})$ for $n$-tuples $\mathbf{z}$ of distinct complex numbers. Define $\cM(\bm\theta)$ to be the smallest affine invariant submanifold containing $\cM_{cyc}(\bm\theta)$. \ann{R2: The definition of $\cM_{cyc}(\theta)$ should be clarified. How are the $n$-tuples $z$ constrained by the $\theta$? The dependence should be clarified because it becomes an issue in the proof of Lemma 6.3. In Line 4 of the proof, it says  ``an open set of $\bR^n$".  However, if $n = 3$, then if the 3-tuples z are subject to the constraints of $\theta$, then the line would read  ``an open set of a subspace $\bR \subset \bR^3$"  in that case.\\ A: There is no restriction on $z_i$ in the definition of $\cM_{cyc}$, other that they be distinct.}

Recall that the set of polygons with given angles is a smooth manifold with a well defined Lebesgue measure class. 

\begin{lem}\label{C:MP}
For almost every $n$-gon $P$ with interior angles $\theta_1, \ldots, \theta_n$, the unfolding of $P$ has $GL(2,\bR)$ orbit closure $\cM(\bm\theta)$.
\end{lem}

\begin{proof}
Let $\cP$ be a connected component of the space of $n$-gons with angles $\theta_1, \ldots, \theta_n$. The theory of Schwarz-Christoffel mappings gives that the set of unfoldings of\ann{A:Fixed typo found by R2.} polygons in $\cP$ is  equal to the set of $(X_{\mathbf{z}}, \omega_{\mathbf{z}})$ where the $z_i$ range over an open subset of $\bR^n$, see for example \cite{DT}.

Recall that a holomorphic function on a connected open subset of $\bC^n$ that vanishes on a nonempty open subset of $\bR^n\subset \bC^n$ must be identically zero. The relative periods of $\omega_{\mathbf{z}}$ are holomorphic functions of the $z_i$, so any linear equation that holds on the periods of all unfoldings of polygons in $\cP$ also holds for all $(X_{\mathbf{z}}, \omega_{\mathbf{z}})$.  We conclude that $\cM(\bm\theta)$ is the smallest affine invariant submanifold containing all unfoldings of polygons in $\cP$.

There are only countably many affine invariant submanifolds \cite{EMM, Wfield}, so in particular there are only countably many proper affine invariant submanifolds of $\cM(\bm\theta)$. Any $P\in \cP$ not in their preimage has unfolding with orbit closure  equal to $\cM(\bm\theta)$. 
\end{proof}

$\cM_{cyc}(\bm\theta)$ is locally defined by the $\xi$ eigenspace for $T$:

\begin{lem}
Every $(X_\mathbf{z}, \Sigma_\mathbf{z})$ in $\cM_{cyc}(\bm\theta)$ has a neighbourhood in $\cM_{cyc}(\bm\theta)$ whose image under the period coordinate map $(X', \omega')\mapsto H^1(X_\mathbf{z}, \Sigma_\mathbf{z})$ is an open set in  $H^1_1(X_\mathbf{z}, \Sigma_\mathbf{z})$.
\end{lem}
\ann{R2: Lemma 6.4 and Corollary 6.5: I feel that these could use more justification. For example, why is the dimension of the open set in Lemma 6.4 equal to the dimension of $H^1_1(X_z,\Sigma_z)$?\\A: Added a proof.}
We thank Curtis McMullen for helpful discussions regarding this lemma. Compare to \cite[Corollary 6.8]{McM:braid}.
\begin{proof}
Note that $\xi^{-1} T^*$ fixes $(X_\mathbf{z}, \omega_\mathbf{z})$. We can consider $\xi^{-1} T$ as an element of the mapping class group of the surface with zeros of the Abelian differential marked. In the Teichm\"uller space of the stratum (the set of surfaces in the stratum, equipped with an isotopy class of homeomorphisms from a fixed translation surface taking zeros to zeros), a connected component of the pre-image of  $\cM_{cyc}(\bm\theta)$ is given by a connected component of the set of fixed points of $\xi^{-1} T$, which is evidently linear because the action of this mapping class is of course linear in period coordinates. 
\end{proof}

\begin{cor}\label{C:eigenandconjugate}
The tangent space to $\cM(\bm\theta)$ at a point of $\cM_{cyc}(\bm\theta)$ contains $H^1_1(X_\mathbf{z}, \Sigma_\mathbf{z})+H^1_{-1}(X_\mathbf{z}, \Sigma_\mathbf{z})$.
\end{cor}

Note that while eigenspaces of $T$ define subbundles of $H^1$ over $\cM_{cyc}(\bm\theta)$, they are not defined on $\cM \setminus \cM_{cyc}(\bm\theta)$. \ann{A: Added remark and changed statement of corollary slightly, to address the next comment of R2.}


\begin{cor}\label{C:billiardsfullrank}
If all $\theta_i$ are multiples of $\pi/3$, then $\cM(\bm\theta)$ is full rank. 
\end{cor}

\begin{proof}
Since $T^*$ has order 3 and does not fix\ann{A: Typo found by R2 fixed.} any nonzero Abelian differentials,  $H^1(X_\mathbf{z})$ is the direct sum of the $\xi$ and $\xi^{-1}=\xi^2$ eigenspaces. 
\end{proof}

\begin{lem}\label{L:nothyp}
Suppose $k$ is odd and either $n>4$ or $\bm\theta$ has more than two distinct angles. Then $\cM(\bm\theta)$ is not contained in the hyperelliptic locus.
\end{lem}

\begin{proof}
Suppose $(X, \omega) \in \cM_{cyc}(\bm\theta)$, and $\tau$ is a hyperelliptic involution on $X$. 
The assumption that $k$ is odd guarantees that $\langle T\rangle$ does not contain an involution, and so in particular it does not contain $\tau$. 
Since $\tau$ and $T$ commute,  $\tau$ descends to an involution of $(X, \omega)/\langle T\rangle$ that is a local isometry for the flat metric. If $(X,\omega)$ is the unfolding of a polygon $P$, then  $(X, \omega)/\langle T\rangle$ is the ``pillowcase double" of $P$.

The generic element of $\cM_{0,n}$ does not have an involution when $n>4$. If $n=3$ or $n=4$, the generic element of $\cM_{0,n}$ does have an involution, which exchanges two of the marked points when $n=3$ and exchanges two pairs of marked points when $n=4$. 
\end{proof}

\begin{proof}[Proof of Theorem \ref{T:billiards}.]
By Corollary \ref{C:billiardsfullrank}, $\cM(\bm\theta)$ has full rank whenever the angles of $P$ are all multiples of $\pi/3$. By Lemma \ref{L:nothyp}, it is not contained in the hyperelliptic locus if the genus is greater than 2. Hence Theorem \ref{T:main} gives that $\cM(\bm\theta)$ is a connected component of a stratum, and Corollary \ref{C:MP} concludes the proof. 
\end{proof}

\section{The instability of eigenform loci}\label{S:deriv}
\ann{R2: Throughout this section many decompositions of $H^1$ are considered. In particular, there are decompositions of the bundle $p(T(\cM))$. Superficially, this appears to contradict the theorem of Alex Wright that this is a simple subbundle of $H^1$. Of course, there is no contradiction here because the decomposition is over elements of $\cM_{cyc}$, which is not $SL_2(\bR)$-invariant. However, this does not appear to be clarified at any point and I think it would be valuable to do so because $SL_2(\bR)$-invariance verses non-invariance is a key component throughout all of the arguments in this section. \\A: Added a comment following Corollary 6.5.}

The goal of this section is to improve the following ``trivial rank bound" in some cases.  We introduce the ideas gradually as they are required for applications rather than proceeding immediately to the strongest possible statements. 

\begin{lem}\label{L:trivialrankbound}
If at least one $\theta_i$ is not a multiple of $\frac\pi2$, the rank of $\cM(\bm\theta)$ is at least $n-2$. If all the $\theta_i$ are multiples of $\frac\pi2$, the rank is at least $(n-2)/2$.
\end{lem}

\begin{proof}
By Corollary \ref{C:eigenandconjugate}, the tangent space of $\cM(\bm\theta)$ at a point of $\cM_{cyc}(\bm\theta)$\ann{A: Added ``at a point of $\cM_{cyc}(\bm\theta)$"} contains the $\xi$ and $\xi^{-1}$ eigenspaces for $T$. If $k=2$, these are equal. 
By Corollary \ref{C:dim}, the $\xi$ eigenspace has dimension $t(1)+t(-1)-2$. If $\ell$ is relatively prime to $k$, then $t(\ell)+t(-\ell)=n$. 
\end{proof}

For the remainder of this section we assume $k>2$. We have written the section to illustrate the relevance of the techniques for arbitrary $n\geq3$, however at key places we have assumed $n=3$ to avoid technical difficulties, and so ultimately we only give applications to orbit closures of unfoldings of triangles. The results of this section are partially inspired by \cite[Theorem 7.5]{Mc}, which exhibits a locus of genus 3 eigenforms for real multiplication by an order in a totally real cubic field that is not $GL(2, \bR)$ invariant.

\subsection{Showing rank bigger than $n-2$.}


For any affine invariant submanifold $\cN$, let  $H^1$ denote the flat bundle over $\cN$ whose fiber over $(X,\omega)$ is $H^1(X)$. This bundle can be decomposed as a direct sum of symplectically orthogonal flat subbundles  
$$H^1=\bigoplus_\rho p(T\cN)^\rho \bigoplus_{s}  \bW_s,$$
where $\rho$ runs over the different field embeddings of the affine field of definition $\bk(\cN)$,  the $p(T\cN)^\rho$ are the Galois conjugates of the bundle $p(T\cM)$, and the $\bW_s$ are the remaining isotypic components for the monodromy representation \cite{Wfield}. 
By work of Filip, this decomposition is a decomposition of variations of Hodge structures, which means that the fibers of each bundle are equal to the direct sum of their\ann{A: Typo found by R2 fixed.} intersections with $H^{1,0}(X)$ and $H^{0,1}(X)$ \cite{Fi2}. 

Define, for any triple of Abelian differentials $\omega, \alpha, \beta$ on a Riemann surface $X$, 
$$B_\omega(\alpha, \beta)=\frac{I}2 \int_X \alpha \beta \frac{\overline{\omega}}{\omega}.$$\ann{A: Changed $i$ to $I$.} 
The expression computes the second fundamental form for the Hodge bundle $H^{1,0}$ and the derivative of the period matrix  \cite[Lemma 2.1, Lemma 2.2]{FMZ:subbundles}. 

\begin{lem}\label{L:B0}
Let $\cN$ be an affine invariant submanifold, and let $(X,\omega)\in \cN$. If $\alpha, \beta\in H^{1,0}(X)$ lie in different summands of the above decomposition of $H^1$, then $B_\omega(\alpha, \beta)=0$. 
\end{lem}


This follows from the\ann{A: Typo found by R2 fixed.} fact that $B$ computes the second fundamental form of $H^{1,0}$, and that the above decomposition  respects the Hodge filtration. See \cite[Section 2.3]{FMZ:subbundles} for the definition of the second fundamental form. 

Recall that by Lemma  \ref{L:eigendim}, the eigenspace $H^{1}_a$ contains nonzero holomorphic one-forms exactly when $t(-a)>1$. 

\begin{prop}\label{P:neq0}
Let $0< a,b <k$ with $t(-a)>1$ and $t(-b)>1$. 
\begin{enumerate}
\item If $a+b\neq 2 \mod k$, then for any $(X,\omega)\in \cM_{cyc}(\bm\theta)$, $\alpha \in H^{1,0}_a(X),$ and  $\beta \in H^{1,0}_b(X)$ we have $B_\omega(\alpha, \beta)=0$.
\item If $n=3$ and $a+b= 2 \mod k$ there  exists $(X,\omega)\in \cM_{cyc}(\bm\theta)$, $\alpha \in H^{1,0}_a(X),$ and  $\beta \in H^{1,0}_b(X)$ such that $B_\omega(\alpha, \beta)\neq 0$.
\end{enumerate}
\end{prop}

The proof is  deferred to Subsection \ref{SS:neq0}. It is plausible that the second statement is also true when $n>3$, and we have checked this in some but not all cases. 

\begin{prob}
Determine precisely when the second statement of Proposition \ref{P:neq0} is true  when $n>3$.
\end{prob}

Because of\ann{A: Typo found by R2 fixed.} Proposition \ref{P:neq0}, we will be interested in pairs of eigenspaces $H^1_a, H^1_{2-a}$ when $t(-a)>1$ and $t(-(2-a))>1$, i.e. when both of these eigenspaces contain nonzero holomorphic one-forms. We will generally want to assume that the eigenspaces $H^1_a$ and $H^1_{2-a}$ are distinct, i.e. that $a\neq 2-a \mod k$.

\begin{thm}\label{T:unstab}
Suppose $n=3$, and suppose  there exists $a\in (\bZ/k)^*$ with $2a\neq 2 \mod k$ and $t(-a)>1$ and $t(-(2-a))>1$. Then $\cM(\bm\theta)$ has rank strictly greater than $n-2$. 
\end{thm}

If Proposition \ref{P:neq0} is true when $n>3$, then so is Theorem \ref{T:unstab}. 
\ann{R2: Proof of Theorem 7.5: Please clarify why $a\neq 2 - a \mod k$ implies that $\omega_a$ and $\omega_{2-a}$ lie in different $SL_2(\bR)$-invariant subbundles of $H^1$.\\A: Added explanation.}
\begin{proof}
If $\cM(\bm\theta)$ has rank $n-2$, then the proof of Lemma \ref{L:trivialrankbound} gives that $p(T_{(X,\omega)}\cM(\bm\theta))$ is equal to  $H^1_1(X)\oplus H^1_{-1}(X)$ at any $(X, \omega)\in T\cM_{cyc}(\bm\theta)$. Similarly, each Galois conjugate of $p(T_{(X,\omega)}\cM(\bm\theta))$ is of the form $H^1_a(X)\oplus H^1_{-a}(X)$ for some $a\in (\bZ/k)^*$. 

Since $t(-a)>1$ and $t(-(2-a))>1$, we can pick nonzero $\omega_a\in H^{1,0}_a(X)$ and $\omega_{2-a}\in H^{1,0}_{2-a}(X)$.  

Since $2a\neq 2 \mod k$, the forms $\omega_a$ and $\omega_{2-a}$  lie in different eigenspaces for $T$. Since by assumption $k>2$, we have $2-a\neq -a \mod k$, so in fact $\omega_{2-a}$ does not lie in the Galois conjugate of $p(T_{(X,\omega)}\cM(\bm\theta))$ that contains $\omega_a$. 

By Proposition \ref{P:neq0}, we get that $B_\omega(\omega_a, \omega_{2-a})\neq 0$, and so Lemma \ref{L:B0} gives a contradiction. 
\end{proof}

The proof can be viewed as showing that the $SL(2, \bR)$ orbit of a surface in $\cM_{cyc}$\ann{A: Typo found by R2 fixed.} is not even tangent to the locus of $(X,\omega)$ admitting the endomorphisms which would be present if  $\cM(\bm\theta)$ had rank $n-2$. These endomorphisms are described in \cite[Theorem 7.3]{Fi2}. 

%

\subsection{Full rank examples.}

We proceed to the proof of Theorem \ref{T:triangles}, for which we may pick the triangles to be rather special. Afterwards  we continue our analysis so that it applies to as many triangles as possible, leading to a simple algorithm which can often show that the unfolding of a triangle has full rank. We continue to assume $n=3$, even though much of our analysis also applies when $n>3$. 

\ann{R2: Lemma 7.6: Consider rephrasing to say  ``Then there exists a multiplicative subgroup $A \subset (\bZ/k)^*$ containing $-1$ such that ...'' \\ A: Ok, we made this change. }
\begin{lem}\label{L:alternatives}
Consider $(X, \omega)\in \cM_{cyc}(\bm\theta)$. Then there exists a multiplicative subgroup $A\subset (\bZ/k)^*$ containing $-1$ such that
   $p(T_{(X, \omega)}\cM(\bm\theta))=V\oplus V'$, where 
$$V=\bigoplus_{a \in A} H^1_{a}(X),$$
  and $V'$ is orthogonal to the direct sum  $\bigoplus_{a\in (\bZ/k)^*}H^1_{a}(X)$ of the primitive eigenspaces. The Galois conjugates of $V$ are given by the direct sums of $H^1_{a}(X)$ over cosets of $A$ in $(\bZ/k)^*$.
\end{lem}

\begin{proof}
Let $A$ be the stabilizer for the action of the Galois group $(\bZ/k)^*$ of $\bQ[\exp(2\pi I /k)]$\ann{A: Changed $i$ to $I$.} on the affine field of definition $\bk$ of $\cM(\bm\theta)$. Note $A$ contains $-1$ since $\bk\subset \bR$. 

The projection of the tangent space of $\cM(\bm\theta)$ to absolute cohomology contains $V=\bigoplus_{a \in A} H^1_{a}(X),$ because any subspace defined over $\bk$ is stable under any field automorphism fixing $\bk$.

Since $p(T_{(X, \omega)}\cM(\bm\theta))$ is orthogonal to all its Galois conjugates,  it is  orthogonal to $\bigoplus_{a \notin A} H^1_{a}(X)$. It follows that $p(T_{(X, \omega)}\cM(\bm\theta))$ is the direct sum of $\bigoplus_{a \in A} H^1_{a}(X)$ and its intersection with $\bigoplus_{a \in (\bZ/k)\setminus (\bZ/k)^*} H^1_{a}(X)$.
\end{proof}\ann{R2:  ``The number of primitive roots mod $k$ goes to infinity as $k \to \infty$".  While this is a very strong reason to believe that the claim following this phrase has to be true, it is not a sufficiently rigorous mathematical justification that such a solution always exists.\\A: Added explanation in parentheses.}


\begin{proof}[Proof of Theorem \ref{T:triangles}.]
Pick $k>2$ to be prime. Let $r$ be a primitive root mod $k$, and let $a$ be the solution to $a(2-a)^{-1}=r \mod k$. It follows that $2a\neq 2 \mod k$\ann{A: Typo fixed by R2 fixed.}. Set 
\begin{align*}
q_1&=(k-1) (-a)^{-1}&\mod k,\\ 
q_2&=(k-1) (-(2-a))^{-1}&\mod k,\\ 
q_3&=k-q_1-q_2&\mod k.
\end{align*}
The number of primitive roots mod $k$ goes to infinity as $k\to\infty$, so  $r$ can be chosen so that $q_i$ are nonzero and distinct. (The equation $q_3=0$, and each equation $q_i=q_j$, determines only boundedly many  $r$ values, since these are polynomial equations of bounded degree over a field.) Since the $q_i$ are distinct, by Lemma \ref{L:nothyp} the unfolding $(X,\omega)$ of the triangle with angles $\frac{q_i}k\pi$  is not hyperelliptic.

Note $t(-a)>1$ since $t_1(-a)= \frac{k-1}{k}$, and all $t_i(-a)\geq \frac1k$. Similarly $t(-(2-a))>1$. 

Let $A$ be as in Lemma \ref{L:alternatives}. Pick nonzero $\omega_a \in H^{1,0}_a(X)$ and $\omega_{2-a} \in H^{1,0}_{2-a}(X)$. By Proposition  \ref{P:neq0}, $B_\omega(\omega_a, \omega_{2-a})\neq 0$, so by  Lemma \ref{L:B0} it follows that $r\in A$, since otherwise $H^1_a(X)$ and $H^1_{2-a}(X)$ would belong to different Galois conjugates of $p(T_{(X, \omega)}\cM(\frac{q_1}k\pi, \frac{q_2}k\pi, \frac{q_3}k\pi))$. 

Since $r$ is a primitive root, we get that $A=(\bZ/k)^*$ and hence that $\cM(\frac{q_1}k\pi, \frac{q_2}k\pi, \frac{q_3}k\pi)$ has full rank. Since $(X, \omega)$ is not hyperelliptic, Theorem \ref{T:main} gives that $\cM(\frac{q_1}k\pi, \frac{q_2}k\pi, \frac{q_3}k\pi)$ must be a connected component of  a stratum. 
\end{proof}

We continue our analysis so that it will apply  more broadly. 

\begin{lem}\label{L:TinMon}
For any $(X,\omega)\in \cM_{cyc}(\bm\theta)$, 
$p(T_{(X, \omega)}\cM(\bm\theta))$ respects the direct sum decomposition of cohomology as $\bigoplus_{a \in \bZ/k} H^1_{a}(X)$, in that it is equal to the direct sum of its intersections with the summands. 
\end{lem}

The idea of the proof can be seen already for the unfolding of the right angled triangle with smallest angle $\frac\pi8$, which is the regular octagon. Gradually rotating the octagon by $\frac{2\pi}8$ gives a path in the stratum from the octagon to itself, whose monodromy is a deck transformation that generates the cyclic deck group of the map to $\bP^1$. 

\begin{proof}
Recall that $\cM(\bm\theta)$ was defined to be the smallest affine invariant submanifold containing $\cM_{cyc}(\bm\theta)$. Since the self crossing locus of any affine invariant submanifold is a smaller dimensional affine invariant submanifold, we may conclude that $\cM_{cyc}(\bm\theta)$ is not contained in the self crossing locus of $\cM(\bm\theta)$. Pick $(X, \omega) \in \cM_{cyc}(\bm\theta)$ that is not contained in the self crossing locus of $\cM(\bm\theta)$.

Recall that  $\omega \in H^1_1(X)$, which means exactly that $T^*(\omega)=\exp(2 \pi I/k) \omega$. For any translation surface $(X, \omega)$, the surface $(X,\exp(2 \pi I/k) \omega)$ is obtained by rotating $(X,\omega)$ by angle $2\pi/k$. \ann{R2: Please use only one symbol for $\sqrt{-1}$ throughout the paper.\\A: We have tried to make it consistent now. Here we changed one $i$ to an $I$.}

$p(T_{(X, \omega)}\cM(\bm\theta))$ is invariant under the monodromy of any path in $\cM(\bm\theta)$. The path $(X, \exp({I\phi})\omega), \phi\in [0, 2\pi/k]$ has monodromy $T$, and so $p(T_{(X, \omega)}\cM(\bm\theta))$ is invariant under  $T$.

The result now follows from the fact that any subspace of a vector space invariant under a finite order linear transformation $T$ is equal to the direct sum of its intersections with the eigenspaces of $T$. 
\end{proof}

\begin{rem}
It was necessary to pick $(X, \omega)$ not in the self crossing locus because if $(X, \omega)$ is in the self crossing, then the image of $\cM(\bm\theta)$ in the stratum is locally defined at $(X, \omega)$ by a finite union of vector subspaces, which might be permuted by $T$. Note that, although $X$ is an orbifold point of $\cM_g$, $(X, \omega)$ is typically not an orbifold point of the stratum, because the automorphisms $T^i$ of $X$ typically do not satisfy $(T^i)^*(\omega)=\omega$. \ann{R2:  ``Note that...".  This is a very insightful comment. Consider elaborating on it.\\A: We elaborated on it.}
\end{rem}

%
%

\begin{cor}\label{C:sum}
When $n=3$, $p(T_{(X, \omega)}\cM(\bm\theta))$ is the direct sum of eigenspaces for any $(X, \omega)\in \cM_{cyc}(\bm\theta)$. 
\end{cor}\ann{R2: Please add the dependence on $\cM_{cyc}$ to the corollary.\\A: Added.}

\begin{proof}
When $n=3$,  nonzero eigenspaces are 1 dimensional. 
\end{proof}

The monodromy of the flat bundle $H^1_a$ over $\cM_{cyc}(\bm\theta)$ is typically (and possibly always) irreducible (\cite[Theorems 3.3.4, 5.1.1, Proposition 5.5.1]{Rohde}, see also \cite[Section 4]{Loo}, \cite[Corollary 5.3]{McM:braid}), which would imply that Corollary \ref{C:sum} applies equally well when $n>3$. Since we haven't found a reference stating explicitly that the monodromy is always irreducible, and since we are happy to continue to assume $n=3$, we haven't attempted to check whether there are exceptional cases where the monodromy isn't irreducible. 

Together with Corollary \ref{C:sum}, the following slight extension of Lemma \ref{L:alternatives} will underlie all of our remaining analysis. 

\begin{lem}\label{L:ab}
Let $A\subset (\bZ/k)^*$ be as in Lemma \ref{L:alternatives}. Let $(X, \omega)\in \cM_{cyc}(\bm\theta)$.

Suppose $H^1_a(X) \subset p(T_{(X,\omega)}\cM(\bm\theta))$, and $b\in (\bZ/k)^*$. If $b\in A$ then $H^1_{ab}(X)\subset p(T_{(X,\omega)}\cM(\bm\theta))$,\ann{R2: Please change $\cM$ to $\cM(\theta)$ in each of these lines for clarity.\\A: Done.} and otherwise $B_\omega(\omega_a, \omega_{ab})=0$ for all $\omega_a\in H^{1,0}_a(X)$ and $\omega_{ab}\in H^{1,0}_{ab}(X)$. 
\end{lem}

\begin{proof}
The first statement follows as in the proof of Lemma \ref{L:alternatives} because $p(T\cM)$ is invariant under field automorphisms of the affine field of definition $\bk$,\ann{A: Corrected a typo by changing $\bC$ to $\bk$.} and the second statement follows from Lemma \ref{L:B0}. 
\end{proof}

\begin{lem}\label{L:dorth}
Suppose $(X,\omega)\in \cM_{cyc}(\bm\theta)$\ann{A: Added ``in $\cM_{cyc}$".} and $d_1, d_2$ are  divisors of $k$ such  that 
\begin{enumerate}
\item $p(T_{(X, \omega)}\cM(\bm\theta))$ contains $H^1_a$ for some $a$ with $\gcd(a,k)=d_1$ and  $H^1_a$ nonzero, and
\item $p(T_{(X, \omega)}\cM(\bm\theta))$ does not contain $H^1_b$ for any $b$ with $\gcd(b,k)=d_2$ and  $H^1_b$ nonzero.
\end{enumerate}
Then for any $a',b'$\ann{R2: Please change $a$ and $b$ in the conclusion to $a'$ and $b'$ for clarity of the statement.\\A: Done.} with $\gcd(a',k)=d_1, \gcd(b',k)=d_2$, and any $\alpha\in H^1_{a'}, \beta\in H^1_{b'}$, $B_\omega(\alpha, \beta)=0$. 
\end{lem}\ann{R2:  ``and its Galois conjugates".  Consider adding the phrase
 ``by the same arguments as in the proof of Lemma 7.10". \\
A: We added this phrase.}

\begin{proof}
By the same arguments as in the proof of Lemma \ref{L:ab}, in this case $\bigoplus_{\gcd(a,k)=d_1} H^1_a$ is contained in the sum of $p(T_{(X, \omega)}\cM(\bm\theta))$ and its Galois conjugates, and  $\bigoplus_{\gcd(b,k)=d_2} H^1_b$ is contained in the complement, so the result follows by Lemma \ref{L:B0}.
\end{proof}

It is helpful to group the eigenspaces $H^1_a$ according to how primitive the associated roots of unity are. This gives the following direct sum decomposition of $H^1$ defined over $\bQ$
$$H^1=\bigoplus_{d|k} \left( \bigoplus_{\gcd(a,k)=d} H^1_a\right).$$
By Corollary \ref{C:dim},  the subspace $\bigoplus_{\gcd(a,k)=d} H^1_a$ is nonzero if and only if there exists an $a$ with $\gcd(a,k)=d$ and $t(-a)>1$.

Define $\cD$ to be the set of divisors $d$ of $k$ such that there exists an $a$ with $\gcd(a,k)=d$ and $t(-a)>1$.
Consider the equivalence relation\ann{R2: Please justify that this is an equivalence relation. For example, it isn't clear that $d \sim d$ because it is not obvious that there would exist an $a$ such that $\gcd(a, k) = \gcd(2 - a, k)$.\\A: An equivalence relation \emph{generated by} some relations is an equivalence relation by definition; see for example ``Generating equivalence relations" on wikipedia.} $\sim$ on $\cD$ generated by $d_1\sim d_2$\ann{A: Typo found by R2 fixed.} if there is exists $a$ with $\gcd(a,k)=d_1, \gcd(2-a,k)=d_2$ and $t(-a)>1$, $t(-(2-a))>1$. 

\begin{lem}\label{L:sim}
If $d\sim 1$ then there is an $b$ with $\gcd(k,b)=d$ and $t(-b)>1$ and $H^1_b\subset p(T_{(X, \omega)}\cM(\bm\theta))$. Furthermore $A$ contains $$\{a\in (\bZ/k)^*: a = 1 \mod k/d\}.$$
\end{lem}

\begin{proof}
Suppose $a\in \bZ/k$ with $\gcd(a,k)=d_1, \gcd(2-a,k)=d_2$ and $t(-a)>1$, $t(-(2-a))>1$, and suppose there exists an $a'\in \bZ/k$ with $t(-a')>1$ and $\gcd(a',k)=d_1$ and $H^1_{a'}\subset p(T_{(X, \omega)}\cM(\bm\theta)).$ Let $\omega_a\in H^{1,0}_a$ and $\omega_{2-a}\in H^{1,0}_{2-a}$. By Proposition\ann{A: Typo found by R2 fixed.} \ref{P:neq0}, $B_\omega(\omega_a, \omega_{2-a})\neq 0$, so Lemma \ref{L:dorth} gives that $p(T_{(X, \omega)}\cM(\bm\theta))$  contains $H^1_b$ for some $b$ with $\gcd(b,k)=d_2$ and  $H^1_b$ nonzero.

Using the fact that $H_1^1\subset p(T_{(X, \omega)}\cM(\bm\theta))$ and the definition of the equivalence relation, the first claim follows. 

The second statement follows from the first statement and Lemma \ref{L:ab}, because for $a\in (\bZ/k)^*$ with  $a=1 \mod k/d$ and  any $b$ with $\gcd(b,k)=d$, $H^1_{ab}=H^1_b$. 
\end{proof}

\begin{lem}
Suppose $1\sim d$ and $\gcd(a,k)=d$,  $2-a=w a \mod k$ for some $w\in (\bZ/k)^*$, and $t(-a)>1, t(-(2-a))>1$. Then $w\in A$. 
\end{lem}

The assumptions imply $\gcd(2-a,k)=d$ and that $d$ equals 1 or 2. 

\begin{proof}
Pick $\omega_a\in H^{1,0}_a$ and $\omega_{2-a}\in H^{1,0}_{2-a}$ nonzero. By Proposition \ref{P:neq0}, $B_\omega(\omega_a,\omega_{2-a})\neq 0$. 

In order to find a contradiction, assume $w\notin A$. Since $1\sim d$,  there is a $b$ with $\gcd(b,k)=d$ and $H^1_b$ nonzero and contained in $p(T\cM)$. This is a contradiction, since in this case $H^1_a$ and $H^1_{2-a}$ are in different Galois conjugates of $p(T\cM)$. 
\end{proof}

\begin{thm}\label{T:fullrank}
Suppose $1\sim d$ for all $d\in \cD$, and $A=(\bZ/k)^*$. Then $\cM$ is full rank. 
\end{thm}

See Appendix \ref{A:alg} for\ann{A: Typo found by R2 fixed.} an algorithm that determines when Theorem \ref{T:fullrank} applies. 

\ann{R2: Section 7.3. ... I do not see why Lemma 7.15 is its own lemma. ... \\ A: Our individual preference for exposition is to separate  proofs into pieces where possible, and so we prefer to leave Lemma 7.15.}

\begin{proof}
By Corollary \ref{C:sum}, over $\cM_{cyc}(\bm\theta)$ the bundle $p(T\cM(\bm\theta))$ is a direct sum of  eigenbundles $H^1_a$. By Lemma \ref{L:sim}, for each $d$ with $\bigoplus_{\gcd(a,k)=d} H^1_a$ nonzero, there is some $b\in \bZ/k$ with $\gcd(b,k)=d$ and $H^1_b$ nonzero and contained in $p(T\cM(\bm\theta))$. By Lemma \ref{L:ab}, we get that  $H^1_{ab}$ for all $a\in A=(\bZ/k)^*$. It follows that $p(T\cM(\bm\theta))$ contains all eigenbundles $H^1_a$. 
\end{proof}

\subsection{Proof of Proposition \ref{P:neq0}.}\label{SS:neq0}

For $0< a<k$ with $t(-a)>1$ and $(X_\mathbf{z}, \omega_\mathbf{z}) \in \cM_{cyc}(\bm\theta)$, recall  Lemma \ref{L:eigendim} and define $$\omega_a=\prod_{i=1}^n (z-z_i)^{-t_i(-a)} dz \in H^{1,0}_a(X_\mathbf{z}).$$ 
Lemma \ref{L:eigendim} states that any element of $H^{1,0}_a(X_\mathbf{z})$ can be written as $p_a(z) \omega_a$, where $p_a(z)$ is a polynomial of degree at most $t(-a)-2$. Recall also that $\omega_\mathbf{z}\in H^{1,0}_1(X_\mathbf{z})$.

\begin{proof}[Proof of Proposition \ref{P:neq0}.]
We are considering expressions of the form $B_{\omega_\mathbf{z}}(p_a(z) \omega_a, p_b(z) \omega_b)$. If we define $p(z)=p_a(z)p_b(z)$, this is   equal to $$\int p(z)\omega_a \omega_b \frac{\overline{\omega_\mathbf{z}}}{\omega_\mathbf{z}}= \frac1k\int p(z) \sum_{\ell=1}^k (T^*)^\ell\left(\omega_a \omega_b \frac{\overline{\omega_\mathbf{z}}}{\omega_\mathbf{z}}\right).$$

If $a+b\neq 2 \mod k$, this integral is  zero because $T$ acts on $\omega_a \omega_b \frac{\overline{\omega_\mathbf{z}}}{\omega_\mathbf{z}}$ by a root of unity other than 1. 
This proves the first claim.

If $a+b= 2 \mod k$, the integrand is $T$ invariant, and hence the integral descends to an integral on $X_\mathbf{z}/\langle T\rangle =\bP^1$, which we will now give an expression for.    \ann{R2: Please reorganize this paragraph. ...
\\A: We reorganized the paragraph as suggested. 
}

\ann{R2: Setting $n = 3$ and then continuing to write n sporadically throughout the rest of the proof was not a good decision. ...
\\ A: We moved the $n=3$ assumption to the end, and instead used the $\theta_i < \pi$ assumption at this point, as you suggested.
}

For the remainder of the proof we assume $\theta_i<\pi$ for all $i$. At the final stage of the proof we will assume that $n=3$, which implies $\theta_i<\pi$ since $\theta_1+\theta_2+\theta_3=\pi$ when $n=3$. Assume $b=2-a \mod k$. Compute


$$\omega_a \omega_b = \prod_{i=1}^n (z-z_i)^{-t_i(-a)-t_i(-b)} (dz)^2 = \omega_\mathbf{z}^2\prod_{i=1}^n (z-z_i)^{\e_i},$$
where 
\begin{eqnarray*}
\e_i&=&2(1-\theta_i/\pi)-t_i(-a)-t_i(-b)\\&=&2-2t_i(1)-t_i(-a)-t_i(-b).
\end{eqnarray*}
With this notation,  
$$\omega_a \omega_b \frac{\overline{\omega_\mathbf{z}}}{\omega_\mathbf{z}}
= \prod_{i=1}^n (z-z_i)^{\e_i} \prod_{i=1}^n |z-z_i|^{2\theta_i/\pi-2} dz d\overline{z}.$$
We require the following numerical observations.  
\ann{R2: Consider the following shorter alternative argument for $\e_i \in
\{-1,0,1\}$:  Since $0 \leq t_i(\ell) < 1$ and $t_i(1) > 0, -2 < \e_i < 2$. ...
\\A: We added your shorter argument, and separated one item into two.}
\begin{enumerate}
\item $ \e_i\in \{-1, 0, 1\}$ for all $i$. This follows since $0\leq t_i(\ell)<1$ and $0<t_i(1)$, using that $\e_i\in \bZ$.
\item $\e_i=-1$ implies $\theta_i\geq \frac\pi2$. This follows similarly, using that  $2t_i(1) = t_i(2)$ if $\theta_i< \frac\pi2$ and $2t_i(1) = 1+ t_i(2)$ otherwise.
\item $\sum \e_i= 4-t(-a)-t(-b).$ Indeed, sum the definition of $\e_i$, and  note that  the sum of the angles in an $n$-gon is $(n-2)\pi$.
\item $4-t(-a)-t(-b)= 0$.  Indeed, by assumption, $t(-a)>1, t(-b)>1$, which implies $t(-a)=2, t(-b)=2$ since $n=3$.
\end{enumerate}
\ann{R2: There are two separate equalities. Write the  first equality, follow with the  first sentence of the argument. Then write the equality   $\e_i = 0$ and write the second sentence. Also please change $t(-a), t(-b) > 1$ to $t(-a) > 1,t(-b) > 1$ and $t(-a),t(-b) = 2$ to $t(-a) = t(-b) = 2$.\\A: We followed all these suggestions.}
Finally, we assume $n=3$. The remainder of the proof follows from the next lemma, which applies after moving one of the $z_i$ to infinity. If the triangle is obtuse, then the\ann{A: Deleted double word.} $z_i$ corresponding to the obtuse angle should be the one moved to infinity. 
\end{proof}

\begin{lem}
Assume $0<\frac{q_1}{k},  \frac{q_2}{k}<\frac12$. With $\e_1, \e_2$ equal to $0$ or $1$, 
$$\int_\bC \frac{  z^{\e_1}(z-1)^{\e_2}  }{  |z|^{2-2\frac{q_1}{k}}|z-1|^{2-2\frac{q_2}{k}}}  dA$$
is  nonzero (assuming the integrand is integrable, in particular that at most one $\e_i$ is nonzero).
\end{lem}

Here $dA$ is the usual area form, which is proportional to $dzd\overline{z}.$ 

\begin{proof}
If $\e_1, \e_2$ are both zero, the result is obvious. If $\e_1=1, \e_2=0$, then the absolute value of the real part of the integrand
$$ \frac{  z }{  |z|^{2-2\frac{q_1}{k}}|z-1|^{2-2\frac{q_2}{k}}}  $$
 at $z$ with $\Re(z)>0$ is greater than that at $-z$. Thus the positive contribution, from when $\Re(z)>0$, dominates the negative contribution from when $\Re(z)<0$. The argument is similar when $\e_1=0, \e_2=1$.
\end{proof}\ann{R2: It is not clear that this argument is complete. From the previous page, it was established that one quotients $X_z$ to $\bP^1$. On the other hand, the integral is only evaluated over $\bC$. Either clarify this discrepancy or explain why the integral about the point at infinity does not cancel the contribution from the  asymmetry  about $0$ that was used in the proof.\\A: $\bP^1\setminus \bC$ has measure 0.}


\newpage

\appendix

\section{Algorithm for Theorem \ref{T:fullrank}}\label{A:alg}\ann{A: Typo found by R2 fixed.}
\begin{algorithm}
\DontPrintSemicolon
\SetKwData{Fac}{$\cD$}
\SetKwData{ConnectedFactors}{$\cE$}
\SetKwData{A}{A}
\SetKwData{done}{$d_1$}
\SetKwData{dtwo}{$d_2$}
\SetKwInOut{Input}{input}
\SetKwInOut{Output}{output}
\Input{$q_1, q_2, q_3, k=q_1+q_2+q_3$, such that $\gcd(q_1, q_2, q_3)=1$ }
\Output{true if Theorem \ref{T:fullrank} implies $\cM(\frac{q_1}k\pi, \frac{q_2}k\pi, \frac{q_3}k\pi)$ has full rank, false otherwise}
\;
\Fac$\leftarrow$ $\{1\}$\;
\ConnectedFactors $\leftarrow$ $\{\}$\;
\A $\leftarrow$ $\{1,-1\}$\;
\;
\For{$a\leftarrow 2$ \KwTo $k-1$}{
  \If{$t(-a)>1$}{\label{lt}
      \done $\leftarrow \gcd(a,k)$\;
      \Fac $\leftarrow$ \Fac $\cup \{\done\}$\;
      \If{$t(-(2-a))>1$}{\label{lt}
        \dtwo $\leftarrow \gcd(2-a,k)$\;
         \uIf{\done $=$ \dtwo}{\label{lt}
           \A $\leftarrow \A \cup \{w \in (\bZ/k)^*: w a=2-a\}$\;
         }
         \Else{\label{lt}
           \ConnectedFactors $\leftarrow$ \ConnectedFactors $\cup \{(\done,\dtwo)\}$\;
         }
      }
    }
}
\;
\If{the equivalence relation on \Fac generated by \ConnectedFactors has more than one equivalence class}{\label{lt}
  \KwRet{false}
}
\;
\For{$d$ in \Fac}{
 \A $\leftarrow \A \cup \{a\in (\bZ/k)^*: a = 1 \mod (k/d)  \}$\:
}
\;
\uIf{\A generates $(\bZ/k)^*$}{\label{lt}
  \KwRet{true}
}
\Else{\label{lt}
  \KwRet{false}
}
\end{algorithm}\DecMargin{1em}
\newpage

\section{List of triangles}\label{A:list}

\begin{figure}[h!]
\begin{tabular}{| c c c c c c  |}
\hline
$(1,2,8)$&
$(1,3,7)$&
$(2,4,5)$&&&\\
\hline
$(1,4,8)$&
$(2,3,8)$&
$(3,4,6)$&&&\\
\hline
$(4,5,6)$&&&&&\\
\hline
$(1,2,14)$&
$(1,4,12)$&
$(1,5,11)$&
$(2,3,12)$&
$(2,4,11)$&
$(2,6,9)$\\
$(2,7,8)$&
$(3,4,10)$&
$(3,6,8)$&
$(4,6,7)$&&\\
\hline
$(1,2,16)$&
$(1,3,15)$&
$(1,4,14)$&
$(1,5,13)$&
$(1,6,12)$&
$(1,7,11)$\\
$(2,3,14)$&
$(2,4,13)$&
$(2,5,12)$&
$(2,6,11)$&
$(2,7,10)$&
$(2,8,9)$\\
$(3,4,12)$&
$(3,5,11)$&
$(3,6,10)$&
$(3,7,9)$&
$(4,6,9)$&
$(4,7,8)$\\
$(5,6,8)$&&&&&\\
\hline
$(1,4,16)$&
$(1,5,15)$&
$(3,5,13)$&
$(4,7,10)$&
$(6,7,8)$&\\
\hline
$(1,2,20)$&
$(1,3,19)$&
$(1,4,18)$&
$(1,5,17)$&
$(1,6,16)$&
$(1,7,15)$\\
$(1,8,14)$&
$(1,9,13)$&
$(2,3,18)$&
$(2,4,17)$&
$(2,5,16)$&
$(2,6,15)$\\
$(2,7,14)$&
$(2,8,13)$&
$(2,9,12)$&
$(2,10,11)$&
$(3,4,16)$&
$(3,5,15)$\\
$(3,6,14)$&
$(3,7,13)$&
$(3,8,12)$&
$(3,9,11)$&
$(4,5,14)$&
$(4,6,13)$\\
$(4,7,12)$&
$(4,8,11)$&
$(4,9,10)$&
$(5,6,12)$&
$(5,7,11)$&
$(5,8,10)$\\
$(6,7,10)$&
$(6,8,9)$&&&&\\
\hline
$(1,3,21)$&
$(1,4,20)$&
$(1,6,18)$&
$(1,7,17)$&
$(1,8,16)$&
$(1,10,14)$\\
$(2,3,20)$&
$(2,4,19)$&
$(2,5,18)$&
$(2,7,16)$&
$(2,8,15)$&
$(2,10,13)$\\
$(3,4,18)$&
$(3,5,17)$&
$(3,6,16)$&
$(3,7,15)$&
$(3,8,14)$&
$(3,9,13)$\\
$(3,10,12)$&
$(4,5,16)$&
$(4,6,15)$&
$(4,8,13)$&
$(4,9,12)$&
$(4,10,11)$\\
$(5,6,14)$&
$(6,7,12)$&
$(6,8,11)$&
$(6,9,10)$&
$(7,8,10)$&\\
\hline
\end{tabular}
\caption{A list of the triples $(q_1, q_2, q_3)$ with $q_1<q_2<q_3$ and $k=q_1+q_2+q_3\leq 25$ odd for which Theorem \ref{T:fullrank} implies that the triangle with angles $(\frac{q_1}{k} \pi,\frac{q_2}{k} \pi,\frac{q_3}{k} \pi)$ unfolds to a surface with dense orbit.}\label{L:list}
\end{figure} 

 \bibliography{mybib}{}
\bibliographystyle{amsalpha}
\end{document}